\documentclass[12pt,reqno]{amsart}
\usepackage{amssymb,mathrsfs,tensor}

\newcommand{\D}{\mathcal{D}}

\newcommand{\F}{\mathcal{F}}

\newcommand{\cS}{\mathcal{S}}

\newcommand{\cR}{\mathcal{R}}
\newcommand{\sS}{\mathscr{S}}
\newcommand{\cG}{\mathcal{G}}

\newcommand{\C}{\mathbb{C}}
\newcommand{\N}{\mathbb{N}}
\newcommand{\R}{\mathbb{R}}

\newcommand{\al}{\alpha}
\newcommand{\be}{\beta}
\newcommand{\ga}{\gamma}
\newcommand{\de}{\delta}
\newcommand{\e}{\varepsilon}
\newcommand{\fy}{\varphi}
\newcommand{\om}{\omega}
\newcommand{\la}{\lambda}
\newcommand{\te}{\theta}
\newcommand{\s}{\sigma}
\newcommand{\ta}{\tau}
\newcommand{\ka}{\kappa}
\newcommand{\x}{\xi}
\newcommand{\y}{\eta}
\newcommand{\z}{\zeta}

\newcommand{\ro}{\rho}

\newcommand{\De}{\Delta}
\newcommand{\Om}{\Omega}
\newcommand{\Ga}{\Gamma}

\newcommand{\Te}{\Theta}

\newcommand{\p}{\partial}
\newcommand{\na}{\nabla}

\newcommand{\re}{\mathop{\mathrm{Re}}}

\newcommand{\sech}{\operatorname{sech}}
\newcommand{\Li}{\operatorname{Li}}

\newcommand{\lec}{\lesssim}
\newcommand{\gec}{\gtrsim}

\newcommand{\I}{\infty}

\newcommand{\ti}{\widetilde}

\newcommand{\ck}{\check}

\newcommand{\U}[1]{\ \underline{#1}\ }
\newcommand{\LR}[1]{{\langle #1 \rangle}}

\newcommand{\EQ}[1]{\begin{equation}\begin{split} #1 \end{split}\end{equation}}

\newcommand{\Br}[1]{\left\{#1\right\}}
\newcommand{\BR}[1]{\left[#1\right]}

\newcommand{\tand}{\ \text{ and }\ }

\newcommand{\Del}[1]{}
\newcommand{\CAS}[1]{\begin{cases} #1 \end{cases}}

\newcommand{\pt}{&}
\newcommand{\pr}{\\ &}
\newcommand{\pq}{\quad}
\newcommand{\pQ}{\qquad}

\newcommand{\pn}{}
\newcommand{\prq}{\\ &\quad}
\newcommand{\prQ}{\\ &\qquad}

\numberwithin{equation}{section}

\newtheorem{thm}{Theorem}[section]
\newtheorem{cor}[thm]{Corollary}
\newtheorem{lem}[thm]{Lemma}

\theoremstyle{remark}

\newcommand{\ox}{\otimes}
\newcommand{\el}{{\bf e}}

\begin{document}
\subjclass[2010]{35J20, 35A23, 35B33, 46E35} 

\title[Maximizing the Trudinger-Moser inequalities]{Sharp threshold nonlinearity for maximizing \\ the Trudinger-Moser inequalities}

\author{S.~Ibrahim}
\address{Department of Mathematics and Statistics, University of Victoria}

\author{N.~Masmoudi}
\address{The Courant Institute for Mathematical Sciences, New York University} 

\author{K.~Nakanishi}
\address{Research Institute for Mathematical Sciences, Kyoto University}

\author{F.~Sani}
\address{Dipartimento di Matematica Federigo Enriques, Universit\`a degli Studi di Milano}

\begin{abstract}
We study existence of maximizer for the Trudinger-Moser inequality with general nonlinearity of the  critical growth on $\R^2$, as well as on the disk. We derive a very sharp threshold nonlinearity between the existence and the non-existence in each case, in asymptotic expansions with respect to growth and decay of the function. The expansions are explicit, using Ap\'ery's constant. 
We also obtain an asymptotic expansion for the exponential radial Sobolev inequality on $\R^2$. 
\end{abstract}

\maketitle

\tableofcontents

\section{Introduction}
The Trudinger-Moser inequality is a well-known substitute for the failed critical Sobolev embedding $H^{1,d}(\Om)\not\subset L^\I(\Om)$ on bounded domains $\Om\subset\R^d$ of $d\ge 2$. The sharp version by Moser \cite{Mo} reads, in the simplest case $d=2$,
\EQ{ \label{oTM}
 u\in H^1_0(\Om),\ \int_\Om |\na u|^2dx \le 4\pi \implies \int_\Om e^{u^2}dx \le C|\Om|,}
for some universal constant $C$, whose optimal value remains unknown to date. The kinetic energy constraint $4\pi$ is chosen in this paper to normalize the nonlinearity, but it can easily be changed to any other constant by a suitable multiple. For example, $u/\sqrt{4\pi}$ satisfies the constraint with $1$, for which the exponential becomes $e^{4\pi u^2}$. 

A natural extension of the above inequality to the case $\Om=\R^2$ was obtained in \cite{TMX}, with a necessary and sufficient condition for general nonlinear energy to be bounded, which reads as follows. 
Let $g:[0,\I)\to\R$ be a continuous function. We have an inequality of the form 
\EQ{ \label{TM on R2}
 u\in H^1(\R^2),\ \int_{\R^2}|\na u|^2 dx\le 4\pi \implies \int_{\R^2}g(u^2)dx \le C\int_{\R^2}|u|^2dx,}
{\it if and only if} $g$ satisfies
\EQ{ \label{TM cond B}
 \limsup_{s\to\I}se^{-s} g(s)<\I \tand \limsup_{s\to0}g(s)/s<\I,}
in other words,
\EQ{
 g(u^2) \lec \left.u^{-2}e^{u^2}\right|_{|u|\gg 1} +  u^2. }

The main question in this paper is for which $g$ a maximizer $u$ exists for the optimal constant $C$ in  \eqref{TM on R2}, which is defined by  
\EQ{
 \pt \cR(g,u):=\|u\|_{L^2(\R^2)}^{-2}\int_{\R^2}g(u^2)dx, 
 \pr \cS(g):=\sup\left\{\cR(g,u) \middle| u\in H^1(\R^2),\ \|\na u\|_{L^2(\R^2)}^2 \le 4\pi \right\}.}
Since the work of Br\'ezis and Nirenberg \cite{BN} on the critical Sobolev embedding of $H^1(\Om)$ for $d\ge 3$, it is well known that existence of maximizer for critical inequalities is subtle and dependent on lower order perturbations of nonlinearity. 
Our analysis in this paper reveals the threshold nonlinearity for $\cS(g)$ asymptotically as $u\to\I$, which is finer than \eqref{TM cond B}, as well as concentrating profile in the case of non-existence. 

For the original Trudinger-Moser inequality \eqref{oTM} in the case of disk $\Om=D$, 
\EQ{
 D:=\{x\in\R^2\mid |x|<1\},} 
the existence of maximizer was proven by Carleson and Chang \cite{CC}. Since then, the same question has been asked and solved for various versions of the Trudinger-Moser inequality. 
In particular, for a well-known version on $\R^2$ by Ruf \cite{Ruf}: 
\EQ{ \label{Ruf}
 u\in H^1(\R^2),\ \int_{\R^2}(|\na u|^2+|u|^2)dx\le 4\pi \implies \int_{\R^2}(e^{u^2}-1)dx \le C,}
the existence of maximizer was also proven in that paper. 

Concerning the non-existence, Pruss \cite{Pru} obtained a general result for perturbations of critical inequalities, which was applied to the Trudinger-Moser, but without precise characterization of the perturbation. 
In a more concrete setting, Ishiwata \cite{I} proved non-existence in the subcritical case of \eqref{Ruf} when $e^{u^2}$ is replaced with $e^{\al u^2}$ for $\al>0$ small enough, by the vanishing loss of compactness.  More recently, Mancini and Thizy \cite{MT} proved non-existence for the Adimurthi-Druet version:
\EQ{ \label{AD}
 u\in H^1_0(\Om), \ \|\na u\|_{L^2(\Om)}^2 \le 4\pi,\ \al<\la_1(\Om) \implies \int_\Om e^{(1+\al\|u\|_{L^2}^2)u^2}dx \le C(\Om)}
by the concentration loss of compactness, when $\al$ is close to $\la_1(\Om)$: the first eigenvalue of the Dirichlet Laplacian. 
In the original setting of \eqref{oTM}, Thizy \cite{T} obtained non-existence by concentration, as well as existence, for perturbed nonlinearity with some sharp conditions, which can be made explicit in the case of the disk.  
Our result given below (Theorem \ref{thm:disk}) can be regarded as an improvement of \cite{T} in the case of disk, concerning the sharp threshold growth. 

Now let us specify the main question in this paper. 
\cite{TMX} also proved that compactness for general sequences of radial functions holds in \eqref{TM on R2} if and only if $g$ satisfies 
\EQ{ \label{TM cond C}
   \lim_{s\to\I}se^{-s} g(s)=0 \tand \lim_{s\to0}g(s)/s=0.}
For the question of maximizer, we may assume the second condition without losing generality if $g$ is differentiable at $0$, because we have
\EQ{ \label{remove mass}
 \cR(g(s)-ms,u)=\cR(g,u)-m}
for any $m\in\R$, and so $\cS(g(s)-ms)$ is attained if and only if $\cS(g)$ is. 

Therefore for existence of maximizer in \eqref{TM on R2}, it remains only to investigate the nonlinearity $g$ with the critical growth for $u\to\I$, namely
\EQ{
   0<\liminf_{s\to\I} se^{-s}g(s) \le \limsup_{s\to\I} se^{-s}g(s)<\I.}
Ignoring the oscillatory case where the middle inequality is strict, we are thus lead to study whether $\cS(g)$ is attained for those continuous $g:[0,\I)\to\R$ satisfying
\EQ{ \label{cond f}
 \lim_{s\to\I}se^{-s}g(s)=1 \tand \lim_{s\to+0}g(s)/s=0.}
Let $\cG$ be the set of all such functions (nonlinearity) $g$, and let
\EQ{
 \pt \cG_M:=\{g\in\cG\mid \exists u\in H^1(\R^2),\ \|\na u\|_{L^2(\R^2)}^2\le 4\pi,\ \cR(g,u)=\cS(g)\}, 
 \pr \cG_N:=\cG\setminus\cG_M,}
be the subsets of those with and without maximizer, respectively. 

Then the critical growth of nonlinearity separating $\cG_M$ and $\cG_N$ turns out to be
\EQ{ \label{sharp thres}
 g(u^2)=u^{-2}e^{u^2}(1-c_Eu^{-4}+O(u^{-6})) \pq(|u|\to\I),}
where $c_E$ is explicitly written using the Riemann zeta function (or Ap\'ery's constant)
\EQ{ \label{def cE}
 c_E:=4+2\z(3).}
A simple example of analytic $g:\R\to\R$ satisfying \eqref{cond f} and \eqref{sharp thres} is 
\EQ{ \label{most critical}
 g(u^2)=\frac{u^2e^{u^2}}{c_E+u^4}.}

For the precise statement, it is convenient to introduce a cut-off for large $u$ of the exponential function: For any $L>0$ and $s>0$, we denote
\EQ{
 \el_L^s :=\CAS{e^s&(s> L^2),\\ 0 &(s\le L^2).}}
Since $s$ in $g(s)$ corresponds to $u^2$, the cut-off for $u$ is at $u=L$ in the above notation. 
The main result of this paper on \eqref{TM on R2} is as follows. 
\begin{thm} \label{thm:main}
We have $\I>\cS(g)\ge\cS_\I:=e^{2-2\ga}$ for all $g\in\cG$, with $\ga$ denoting Euler's constant, where the equality holds if $g\in\cG_N$. There is an absolute constant $C_*>0$ such that the following hold. 
Let $p\in(0,3]$, $q\in(p,\I)$, $a\in(0,\I)$ and $b\in\R$. In the case of $p=3$, let $a\ge C_*$. 
Then for any $L>0$, all $g\in\cG$ satisfying one of the following \eqref{ex by large}--\eqref{ex by small} are  in $\cG_M$ {\rm (Existence):}
\begin{enumerate}
\item $\forall s>0$, $g(s) \ge s^{-1}\el_L^s[1-c_Es^{-2}+a s^{-p}]$. \label{ex by large}
\item $\forall s>0$, $g(s) \ge s^{-1}\el_L^s[1-c_Es^{-2}+b s^{-q}]+as^{1+p}$. \label{ex by small}
\end{enumerate}
On the other hand, if $L>0$ is large enough (depending on $a,b,p,q$), then all $g\in\cG$ satisfying one of the following \eqref{nex by large}--\eqref{nex by small} are in $\cG_N$ {\rm (Non-existence):} 
\begin{enumerate} \setcounter{enumi}{2}
\item $\forall s>0$, $g(s) \le s^{-1}\el_L^s[1-c_Es^{-2}-a s^{-p}]$. \label{nex by large}
\item $\forall s>0$, $g(s) \le s^{-1}\el_L^s[1-c_Es^{-2}+ bs^{-q}]- as^{1+p}$. \label{nex by small}
\end{enumerate}
\end{thm}
In particular, the maximum $\cS(g)$ is attained for the critical nonlinearity $g(u^2)=u^{-2}\el_L^{u^2}$ (satisfying \eqref{ex by large} with $p>2$ and $b=0$), but the existence of maximizer is unstable for lower order perturbations of $O(u^{-6+\e}\el_L^{u^2})$ and of $O(u^{6-\e})$ (for any $\e>0$), respectively by \eqref{nex by large} and \eqref{nex by small} with $p<2$. 
In the more critical case \eqref{most critical}, it is unstable for perturbations of $O(u^{-8+\e}\el_L^{u^2})$ and of $O(u^{8-\e})$. 
The non-existence part with the conditions \eqref{nex by large} and \eqref{nex by small} answers negatively to the question left open in \cite{era}. 

The conditions \eqref{ex by large} and \eqref{nex by large} describe the threshold between $\cG_M$ and $\cG_N$ for lower order perturbations in large $|u|$, to the order $O(u^{-8}e^{u^2})$. 
Note that the leading term in $\cG$ is $u^{-2}e^{u^2}$, so the second order term $u^{-4}e^{u^2}$ is absent in the threshold nonlinearity \eqref{sharp thres}. This absence necessitates the third order expansion detecting the sign on the next term $u^{-6}e^{u^2}$, in order to solve the question in the critical case $g(u^2)=u^{-2}\el_L^{u^2}$. 

The conditions \eqref{ex by small} and \eqref{nex by small} describe the threshold for perturbations in small $|u|$, to the order $O(u^8)$. Note that the term $bs^{-q}$ are giving more room for the condition to hold, which means that perturbation of $O(u^{-2-2q}e^{u^2})$ is dominated by the term $a u^{2+2p}$, in contribution to $\cS(g)$. 

The first sentence of the theorem implies the following ordered structure of $\cG_M$ or $\cG_N$: Let $g_1,g_2\in\cG$ satisfy $g_1(s)\le g_2(s)$ for all $s>0$. If $g_1\in\cG_M$ then $g_2\in\cG_M$ (or equivalently, if $g_2\in\cG_N$ then $g_1\in\cG_N$). Indeed, it is obvious if $\cS(g_2)>\cS_\I$. If $\cS_\I=\cS(g_2)=\cS(g_1)$  and $\fy\in H^1(\R^2)$ maximizes $\cS(g_1)$, then it also maximizes $\cS(g_2)$. 
In short, larger nonlinearity tends to attain the maximum. However, there is no global minimum of $\cG_M$, or global maximum of $\cG_N$, which would be the exact threshold if existed. Actually, the conditions \eqref{ex by small} and \eqref{nex by small} yield concrete examples of $(g_1,g_2)\in\cG_M\times\cG_N$ such that $g_1(s)<g_2(s)$ for large $s$.

A similar result is obtained for the original inequality \eqref{oTM} on $D$, where the threshold nonlinearity is given by 
\EQ{
 g(u^2):=e^{u^2}\BR{1-u^{-2}-c_D u^{-4}+O(u^{-6})}\pq(|u|\to\I),}
with an explicit constant
\EQ{
 c_D:=3/2+2\z(3).}
A simple example of analytic $g$ is 
\EQ{
 g(u^2)=\frac{u^4e^{u^2}}{c_D+1+u^2+u^4}.}

Note that the second order term ($-u^{-2}e^{u^2}$) is present in this case, in contrast to \eqref{sharp thres}. This means that the existence of maximizer is more stable for the original Trudinger-Moser \eqref{oTM} on the disk $D$ than \eqref{TM on R2} on the whole plane $\R^2$ with the critical nonlinearity $u^{-2}e^{u^2}$, which makes the latter problem more delicate. 
It is worth noting, however, that vanishing of the second order term was also observed in the asymptotic expansion by Mancini and Martinazzi \cite{MM} of $\|\na u\|_{L^2(D)}^2$ with respect to $\|u\|_{L^\I}$ for concentrating sequences of critical points for \eqref{oTM}. It does not seem clear if there is any relation to the vanishing observed in this paper on $\R^2$. 

To state the result on the disk in a way parallel to the $\R^2$ case, let 
\EQ{
 \cS^D(g):=\sup\Br{|D|^{-1}\int_D g(u^2)dx \Bigm| u\in H^1_0(D),\ \|\na u\|_2^2\le 4\pi}}
denote the best constant on the disk $D$. In this case, it is natural to assume 
\EQ{ \label{cond f D}
 \lim_{s\to\I} e^{-s}g(s)=1, \pq \lim_{s\to+0} g(s)=0,}
but as in the $\R^2$ case, the left limit can be changed to any positive number by multiplying $g$ with a constant, and the right limit can be changed to any real number by adding a constant to $g$. For the standard nonlinearity $e^{u^2}$, the right limit is $1$. 
\begin{thm} \label{thm:disk}
For any continuous $g:[0,\I)\to\R$ satisfying \eqref{cond f D}, we have $\I>\cS^D(g) \ge e$, where the equality holds if $\cS^D(g)$ is not attained. 
There is an absolute constant $C_*>0$ such that the following hold. 
Let $p\in(0,3]$, $q\in(p,\I)$, $a\in(0,\I)$ and $b\in\R$. In the case of $p=3$, let $a\ge C_*$. 
If $g$ satisfies \eqref{cond f D} and one of the following \eqref{ebl D}--\eqref{ebs D} for some $L>0$, then $\cS^D(g)$ is attained {\rm (Existence):} 
\begin{enumerate}
\item $\forall s>0$, $g(s) \ge \el_L^s[1-s^{-1}-c_Ds^{-2}+a s^{-p}]$. \label{ebl D}
\item $\forall s>0$, $g(s) \ge \el_L^s[1-s^{-1}-c_Ds^{-2}+b s^{-q}]+ a s^p$. \label{ebs D}
\end{enumerate}
On the other hand, if $g$ satisfies \eqref{cond f D} and one of the following \eqref{nbl D}--\eqref{nbs D} for sufficiently large $L>0$, depending on $p,q,a,b$, then $\cS^D(g)$ is not attained {\rm (Non-existence):}
\begin{enumerate} \setcounter{enumi}{2}
\item $\forall s>0$, $g(s) \le \el_L^s[1-s^{-1}-c_Ds^{-2}-a s^{-p}]$. \label{nbl D}
\item $\forall s>0$, $g(s) \le \el_L^s[1-s^{-1}-c_Ds^{-2}+bs^{-q}]-a s^p$. \label{nbs D}
\end{enumerate}
\end{thm}
In the above theorem, the first sentence was already proven by Carleson and Chang \cite{CC}. 
The explicit distinction between the existence and the non-existence was \cite[Open problem 2]{MM}.  The first solution was obtained for general bounded domains in \cite{T}, where the conditions are explicit in the case of disk \cite[Corollary 1.1]{T}. Roughly speaking, it considers nonlinearity in the form 
\EQ{
 0<g(s)=\CAS{e^s[1+c's^{-a'}(\log s)^{-b'}(1+o(1))] &(s\to\I)\\ e^s &(s\ll 1)}}
for some $c',b'\in\R$ and $a'\ge 0$, where $b'>0$ if $a'=0$. 
If $a'>1$ or $c'>0$, then the maximizer exists for all such $g$, while if $a'<1$ and $c'<0$, then the maximizer does not exist for some $g$. 
The existence part is covered by Theorem \ref{thm:disk}, \eqref{ebl D} with $p=1$ and $a=1/2$, for which we do not need to specify even the coefficient of $\el_L^ss^{-1}$, but its sign (which is negative) is enough. 
In this sense, the above result is much sharper about the threshold growth.  
However, the non-existence part is not really covered by \eqref{nbl D} with $p<1$, as it does not allow $g(s)=e^s$ for $s\ll 1$. 

Let us turn to some consequences of the above results on $\R^2$ for the ground state solutions of the nonlinear Schr\"odinger equation:  
\EQ{ \label{sNLS}
 -\De \fy + \om \fy = g'(\fy^2)\fy,\ \fy\in H^1(\R^2),}
where $\om>0$ is a free parameter (time frequency for the evolution). 
Here we assume 

\U{Assumption (A)}
\begin{itemize} 
\item $g:(0,\I)\to\R$ is $C^1$ and $g(s)=o(s)$ as $s\to+0$.
\item $sg'(s)\ge(1+\e)g(s)$ for all $s>0$ with some constant $\e>0$.
\item $g'(s)\le Ce^{Cs}$ for all $s>0$ with some constant $C<\I$. 
\end{itemize}
If the constrained minimization 
\EQ{
 k_\om(g):=\inf\left\{\|\na\fy\|_{L^2(\R^2)}^2 \Bigm| \fy\in H^1(\R^2),\ \int_{\R^2} g(\fy^2)dx=\om\int_{\R^2}|\fy|^2dx>0\right\}}
is attained by some $\fy\in H^1(\R^2)$, then an appropriate rescaling in the form $\fy(\la x)$ solves the above equation \eqref{sNLS}. 
In order to rescale the energy constraint, define $g_\la:[0,\I)\to\R$ for any $\la>0$ by
\EQ{
 g_\la(s)=g(\la s).} 
Then we have $k_\om(g)=\la k_{\la\om}(g_\la)$. 
By definition of $\cS(g)$, together with the above Assumption (A), we have $k_\om(g)>0$ and 
\EQ{ \label{omega range}
 \CAS{\om<\frac1\la\cS(g_\la) \implies & k_{\la\om}(g_\la)<4\pi \iff k_\om(g)<4\pi\la ,\\
 \om=\frac1\la\cS(g_\la) \implies &k_{\la\om}(g_\la)=4\pi \iff k_\om(g)=4\pi\la ,\\
 \om>\frac1\la\cS(g_\la) \implies &k_{\la\om}(g_\la)\ge 4\pi \iff k_\om(g)\ge 4\pi\la .}}
In the first case $\om<\cS(g_\la)/\la$, the kinetic energy level is in the subcritical range, so that $k_\om(g)$ is attained by compactness, cf.,~\cite{TMX}. 
If $\om>\cS(g_\la)/\la$ and $k_\om(g)=4\pi\la$, then $k_\om(g)$ is not attained, by the definition of $\cS(g)$. 
The above result implies that $k_\om(g)$ is not attained in the critical case $\om=\cS(g_\la)/\la$, provided that $g/a$ with some constant $a\in(0,\I)$ satisfies \eqref{cond f} and one of the non-existence conditions \eqref{nex by large} and \eqref{nex by small}. Then the equation \eqref{sNLS} has no solution satisfying 
\EQ{
 \om \ge \tfrac1\la \cS(g_\la) = \tfrac{a}{\la}\cS_\I \tand \|\na\fy\|_2^2 \le 4\pi\la.}
Therefore the correction in \cite{era}, which retracted some claims in the critical case---in particular existence of such a solution in \cite[Theorem 5.1]{TMX}---was indeed necessary for general nonlinearity. 
Of course, \eqref{sNLS} may have a solution with $\om\ge\cS(g_\la)/\la$ and $\|\na\fy\|_2^2>4\pi\la$, namely with supercritical energy, but it is a different issue. 

As for the preceding works in the critical setting, de Figueiredo and Ruf \cite{FR} proved existence of the ground state when $g(s)$ grows like $g(s)\sim s^ae^s$ as $s\to\I$ for $a>-1$. 
Since $a=-1$ is the critical case of \eqref{TM on R2}, we have $\cS(g)=\I$ for such nonlinearity, and so the ground state in the first case of \eqref{omega range}. 
Ruf and Sani \cite{RS} proved the existence in the critical case $g(s)\sim \be_0 s^{-1}e^{4\pi s}$ with sufficiently large $\be_0$ (for a fixed $\om=1$). 
The above argument implies the following consequence of Theorem \ref{thm:main}: 
\begin{cor}
Let $g:[0,\I)\to\R$ satisfy Assumption {\rm (A)} and 
\EQ{
 \lim_{s\to\I} se^{-\la s}g(s)=a}
for some $\la,a>0$. Then \eqref{sNLS} has a positive radial solution $\fy\in H^1(\R^2)$ for 
\EQ{
 0<\om< a\la^2\cS_\I,}
which is a mountain-pass critical point of the energy functional
\EQ{
 E_\om(\fy):=\int_{\R^2}[|\na\fy|^2+\om|\fy|^2-g(\fy^2)]dx.}
On the other hand, there is a function $g$ satisfying the above assumptions for which there is no mountain-pass critical point for any $\om\ge a\la^2\cS_\I$. 
\end{cor}
Here a mountain-pass critical point is any $\fy\in H^1(\R^2)$ satisfying \eqref{sNLS} and 
\EQ{
 E_\om(\fy)=\inf\{\sup_{0\le t\le 1}E_\om(\Ga(t))\mid \Ga\in C^1([0,1];H^1),\ E_\om(\Ga(0))=0>E_\om(\Ga(1))\}.}
The above implies that the optimal condition on $\be_0$ for \cite[Theorem 5]{RS} is 
\EQ{
 \be_0 > \frac{1}{(4\pi)^2\cS_\I} = \frac{e^{2\ga-2}}{16\pi^2}.}

On the other hand, Alves, Souto and Montenegro \cite{ASM} considered a positive $L^p$ ($p>2$) perturbation of the energy and derived a variational lower bound on its coefficient to ensure the existence of ground state as a minimizer of $k_\om(g)$. This is similar to the conditions \eqref{ex by small} and \eqref{nex by small} in the sense that the perturbation is more effective for smaller $|u|$, but it seems difficult to compare those different conditions. 

It is also worth noting that the critical growth $u^{-2}e^{u^2}$ had been known as a threshold for the Dirichlet problem, since the work of de Figueiredo, Miyagaki and Ruf \cite{FMR}. 
In particular, for $g(s)\sim \be_0s^{-1}e^{4\pi s}$ on the disk $\Om=D$, 
de Figueiredo and Ruf \cite{FR} proved non-existence of positive radial solution for $0<\be_0\ll 1$, while de Figueiredo, do \'O and Ruf \cite{FOR} proved existence for $\be_0>\frac{1}{e\pi}$. 
The relation to the $\R^2$ case is yet to be investigated. 

The main strategy to prove Theorem \ref{thm:main} follows the idea of Carleson and Chang \cite{CC}. If a maximizing sequence is concentrating for loss of compactness, then we can evaluate its limit for $\cR(g,u)$. If there is some function exceeding the limit, then the maximum is attained. If all the functions under the condition have less $\cR(g,u)$ than the limit, then the maximum is not attained. 

The latter is more difficult since we need to consider all the candidates. 
So we introduce the cut-off $\el_L$ to large $L\gg 1$, which forces the candidates to concentrate, since otherwise their contribution for the cut-off function is too small.
In order to study the concentrating behavior in detail, the idea is to split the space (radial) region into the central part and the tail part, by partition of the kinetic energy into halves: 
\EQ{
 \|\na u\|_{L^2(|x|<R)}^2 = \|\na u\|_{L^2(|x|>R)}^2=2\pi,}
for any candidate $u$ which is radially decreasing (by the rearrangement). 
Assuming that $L\gg 1$ and $\cR(g,u)\gec 1$, we see that the nonlinear integral is negligible in the tail region $|x|>R$, while the $L^2$ mass is negligible in the central region $|x|<R$, so that we can decompose the maximization problem into two independent ones in those regions. 

Our asymptotic analysis of those two problems are parametrized by the ``initial data" at the middle radius, namely, 
\EQ{
 H:=u(R), \pq a:=x\cdot\na u(R).}
The main novelty is to consider the maximization for each fixed $H\gg 1$, which is a subcritical problem with respect to the Trudinger-Moser inequality, and to analyze the asymptotic behavior as $H\to\I$. 
This argument covers all possible candidates for maximizer, so that we can conclude the non-existence in some cases, after the cut-off forcing $H\gg 1$. 
For the existence proof in the literature, it was enough to consider only some particular $u$ for which $\cR(g,u)$ exceeds the concentration limit, as in \cite{CC,FOR,Ruf}. 
The non-existence proof of Mancini and Thizy \cite{MT} for \eqref{AD} is by a contradiction argument, using asymptotic analysis similar to \cite{MM} for critical points, with respect to the built-in parameter $\al\to\la_1-0$. 
The non-existence proof by Thizy in \cite{T} is by asymptotic analysis on maximizers for subcritical energy in the critical energy limit. 
The advantage of our asymptotic analysis seems that it can easily be linked to the nonlinear growth very precisely.

In the tail region $|x|>R$, the maximization problem becomes linear, which is to derive a sharp form of the exponential radial Sobolev inequality: for any radial $u\in H^1(\R^2)$ and any $R>0$, 
\EQ{ \label{exp radS}
 \|\na u\|_{L^2(|x|>R)}^2\le 4\pi,\ |u(R)|\ge 1 \implies \frac{e^{u(R)^2}}{u(R)^2} \lec \int_{|x|>R}\frac{|u|^2}{R^2}dx,}
proven in \cite{TMX}. The sharp form is the following.
\begin{lem} \label{lem:radSob}
There is an increasing function $\mu:[0,\I)\to[0,\I)$ such that for any radial $u(|x|)=u(x)\in H^1(\R^2)$ and any $R>0$ we have 
\EQ{ \label{sharp radS}
 \mu\left(\frac{4\pi u(R)^2}{\|\na u\|_{L^2(|x|>R)}^2}\right) \le \frac{\|u\|_{L^2(|x|>R)}^2}{R^2\|\na u\|_{L^2(|x|>R)}^2}.}
Moreover, for each $R>0$, there is some $u$ for which the equality holds. In other words, $\mu$ is the optimal (maximal) function for the above inequality. It has the following asymptotic behavior:
\EQ{ \label{exp mu}
 \mu(s)=\CAS{\frac{e^s}{4s}e^{2\ga-1}\BR{1-s^{-1}-\tfrac12s^{-2}+O(s^{-3})}^{-1} &(s\to\I),\\ \frac{1}{16}s^2+O(s^3) &(s\to+0).} } 
\end{lem}
This asymptotic formula of the optimal bound may have interest independent of the Trudinger-Moser inequality. The optimizer $u$ is given by rescaling the Green function of $-\De+1$ on $\R^2$. 
It is also worth noting that the exponential radial Sobolev \eqref{exp radS} follows immediately from the Trudinger-Moser inequality of the exact growth \eqref{TM cond B}, but the sharp asymptotic formula is not easily transferred from the latter to the former. 

In the central region $|x|<R$, we consider maximization of $\cR(g,u)$ for the half energy therein, ignoring the $L^2$ mass, and assuming that the height $H=u(R)$ is optimized by the half energy in the tail $|x|>R$. Using the Euler-Lagrange equation, we see that $-|x|^2\De u$ is approximated by a soliton in the logarithmic coordinate $t=\log(R/|x|)$. The approximating equation and soliton are respectively 
\EQ{
 \ddot u = \frac 2a\dot u(\dot u-a) = \frac a2 \sech^2(t-T_a),}
where the central position $T_a\approx H^2$ is chosen such that $2\dot u(T_a)=a\approx 1/H$. 
The soliton approximation was used already by Carleson and Chang \cite[(18)]{CC}, though in a different scaling. 
The nonlinear integral $\cR(g,u)$ is mostly around $u=2H\approx \|u\|_{L^\I}$, which makes it easy to treat lower order perturbation. Thus it suffices to consider the asymptotic expansions only for the threshold nonlinearity. 
The soliton approximation is enough to obtain the second order expansion, which turns out to be vanishing for $g(s)=s^{-1}\el_L^s$. 
The next order expansion is obtained by using linearization around the soliton. 

The rest of paper is organized as follows. In the next section, we introduce some notation. 
In Section \ref{sect:reduc}, the maximization problem is reduced to an asymptotic formula with the  concentration parameter $H=u(R)$ for $s^{-1}\el_H^s$. 
The asymptotic formula is derived up to the second order in Section \ref{sect:loc int} by using the soliton approximation, and then to the third order in Section \ref{sect:next exp} by using linearization around the soliton. 
In Section \ref{sect:exp Sob}, we prove the asymptotic formula for the exponential radial Sobolev, namely Lemma \ref{lem:radSob}. 
The same analysis on the disk is sketched in Section \ref{sect:disk}. 
Appendix \ref{app:exp int} gathers some explicit integral formulas used in Sections \ref{sect:loc int}--\ref{sect:next exp}, while Appendix \ref{app:asy exp} summarizes the asymptotic expansions in those sections. 
Finally in Appendix \ref{app:R2 2 D}, we derive the sharp Trudinger-Moser inequality on the disk $D$ from the exact version on $\R^2$. 

\section*{Acknowledgements}
S.~I.~is partially supported by {NSERC} Discovery
grant \# {371637-2014}. N.~M.~is in part supported by
{NSF} grant {DMS-1211806}.
K.~N.~was supported by JSPS KAKENHI Grant Number 25400159 and JP17H02854.

\section{Notation}
For any radial function $u=u(r):[0,\I)\to\R$, $g:[0,\I)\to\R$, and $I\subset(0,\I)$,  the radial nonlinear energy, kinetic energy and mass are denoted respectively by 
\EQ{
 \{g\}_{I}(u):=\int_I g(u^2)rdr, \pq K_{I}(u):=\int_I |u'|^2rdr, \pq M_{I}(u):=\int_I|u|^2rdr.}
The subscript $I$ is omitted when $I=(0,\I)$, namely on the entire space. 

For a parameter $h\in\R$, we will denote any bounded and exponentially decaying (as $h\to\I$) quantities by $\ox_h$. 
More precisely, for any function $\al:\R\to\R$, $\al(h)=\ox_h$ means that  
\EQ{
 \al(h) \le Ce^{-\max(h/C,0)}}
for some constant $C>0$ which could be written explicitly, but may vary from place to place. 
If the parameter $h$ is restricted to some range $I\subset\R$, such as $h\ge H$ for some $H\in\R$, then the above inequality is also assumed on the restricted range of $h\in I$. 
Also, we will often replace such $\al(h)$ with $\ox_h$ inside various expressions, in the same way as Landau's symbol $O$.

\section{Reduction to the concentrating half energy}  \label{sect:reduc}
Here we reduce the maximizing problem to the very critical case, namely
\EQ{ \label{def f*}
 g^*_L(s):=s^{-1}\el_L^se^{-c_Es^{-2}} = \CAS{ s^{-1}e^{s}\BR{1-c_Es^{-2}+O(s^{-3})} &(s> L^2),\\ 0 &(s\le L^2),}}
for $L\gg 1$. The second order term $c_Es^{-2}$ has no effect in this or next section. 
The constant $c_E$ is chosen for cancellation in some asymptotic expansion in Section \ref{sect:next exp}. 

First, \eqref{cond f} implies \eqref{TM cond B}, so $\cS(g)<\I$. 
The symmetric rearrangement allows us to restrict $u\in H^1(\R^2)$ to radial decreasing functions. 
Moreover, thanks to the scaling invariance of $\dot H^1(\R^2)$, we may normalize the $L^2$ norm. In other words, 
\EQ{
 \cS(g)\pt=\sup_{u\in X}\{g\}(u), 
 \prq X:=\{u:(0,\I)\to[0,\I)\mid u'\le 0,\ K(u)\le 2,\ M(u)\le 1\}, }
and this reduction does not change the attainability. 
Note that $K(u)\le 2$ is equivalent to the energy constraint $\|\na u\|_{L^2(\R^2)}^2\le 4\pi$.

Now let $L\ge 4$ large enough such that $g^*_L(s)\sim s^{-1}\el_L^s$. 
For any $u\in X$, define $S,\de,R,H\in[0,\I]$ by 
\EQ{ \label{def param}
 \pt S=S(L,u):=\inf\{r>0 \mid u(r)\le L\},
 \pq \de=\de(L,u):=K_{(S,\I)}(u),
 \pr R=R(u):=\inf\{r>0 \mid K_{(r,\I)}(u)\le 1\}, \pq H=H(u):=u(R).}
If $S=0$, then $\el_L^{u^2}\equiv 0$. 
Otherwise we have $0<S<\I$, $u(S)=L$, $u(r)>L$ for $0<r<S$, and $\de\in(0,1]$. 
If $R=0$, then $K(u)\le 1$ and $\sqrt{2}u\in X$, so taking any $\te\in(0,1)$ and using the subcritical Trudinger-Moser, 
\EQ{
 \{\el_L^s\}(u) \le e^{-L^2/3} \{\el_L^{2s/3}\}(\sqrt{2} u) \le \ox_L \cS(\el_L^{2s/3}) = \ox_L.}
Otherwise we have $0<R<\I$, $K_{(R,\I)}=1\ge K_{(0,R)}(u)$, and $0<H<u(r)$ for $0<r<R$. 
Henceforth we assume that $0<S,R,H<\I$. 

Lemma \ref{lem:radSob} yields
\EQ{ \label{bd S 0}
 S^2 \lec (L/\de)^2e^{-2L^2/\de} = L^{-2} (L^2/\de)^2 e^{-2L^2/\de}.}
Hence using $L\ge 4$ and $2\ge \de>0$, we obtain 
\EQ{ \label{bd S}
 \log(1/S) \ge \frac{L^2}{C_1\de}}
for some absolute constant $C_1\in[1,\I)$. 
On the other hand, by the Schwarz inequality, we have 
\EQ{ \label{Schw}
 0<r<s \implies u(r)\le u(s)+\int_s^ru_rdr \le u(s)+\sqrt{K_{(r,s)}(u)\log(s/r)}.}

Let us start with the easy case $L\ge 4H$, where $u$ is spread. Let 
\EQ{
 h:=\max(1,H), \pq \ro:=\inf\{r>0\mid u(r)\le h\}.} 
Then $L\ge 4h$, $S<\ro\le R<\I$, $u(\ro)=h$. Lemma \ref{lem:radSob} implies 
\EQ{
 \ro^2 \le \frac{M_{(\ro,\I)}(u)}{K_{(\ro,\I)}(u)\mu(h^2)} \lec h^2e^{-h^2} \le 1,}
and \eqref{Schw} with $K_{(0,\ro)}(u)\le 1$ implies for $0<r<S$ 
\EQ{
 \Br{\frac34u(r)}^2 \le (u(r)-h)^2 \le \log(\ro/r),}
hence 
\EQ{ \label{small H int}
 \{\el_L^s\}(u) \le \int_0^S e^{u^2}rdr
  \le  e^{-\frac{L^2}{16}}\int_0^S (\ro/r)^{\frac{17}{9}}rdr  
  \lec e^{-\frac{L^2}{16}}\ro^2 \le \ox_L.}
Therefore,
\EQ{ \label{H is big}
 L \gg 1 \tand \{\el_L^s\}(u)\gec 1 \implies H>L/4 \ge 1.}

If $H\ge 1$, then Lemma \ref{lem:radSob} yields $R \lec He^{-H^2}$. 
Hence \eqref{Schw} implies 
\EQ{ \label{int mass}
 M_{(0,R)}(u) \lec \int_0^R (H^2+\log(R/r))rdr \lec R^2(H^2+1)=\ox_H.} 
Let $R':=\min(S,R)$ and $c:=R'/(He^{-H^2})$. Then we have, using \eqref{Schw},
\EQ{
 0<r<R' \implies H^2 \le u(r)^2 \le 2H^2 + 2\log(R'/r) = 2\log(c H/r) ,}
hence
\EQ{ \label{int f bd}
 \{g_L^*\}_{(0,R)}(u) \le \int_0^{R'} \frac{(c H)^2}{4r^2\log^2(c H/r)}rdr =\frac{(c H)^2}{4\log(c H/R')}=\frac{c^2}{4}.}

If $H\le L$, then $\{g_L^*\}_{(0,R)}(u)=\{g_L^*\}(u)$. 
If $H>L$, then $R<S$, $\de<1$, and by \eqref{Schw} we have 
\EQ{ \label{H-L bd}
 (H-L)^2 \le (1-\de)\log(S/R).}
For $R<r<S$, using the elementary inequality
\EQ{
 a,b \in\R \implies (a+b)^2 \le a^2+b^2+\frac{a^2}{1-\de}+(1-\de)b^2=\frac{2-\de}{1-\de}a^2+(2-\de)b^2,}
we obtain
\EQ{
 u(r)^2 \le \frac{2-\de}{1-\de}(u(r)-L)^2 + (2-\de)L^2
 \le (2-\de)\log(S/r)+(2-\de)L^2,}
and so, 
\EQ{ \label{ext f bd}
 \{\el_L^s\}_{(R,\I)}(u)
 \pn= \int_R^S e^{u^2}rdr
 \pt\le \int_R^S (S/r)^{2-\de}e^{(2-\de)L^2}rdr 
 \pr\le \frac{S^2e^{(2-\de)L^2}}{\de} \lec \frac{e^{(2-\de-2/\de)L^2}L^2}{\de^3}=\ox_L,}
where \eqref{bd S 0} was used in the third inequality. 
Combining \eqref{H is big}, \eqref{int f bd} and \eqref{ext f bd}, we deduce that 
\EQ{ \label{est R}
 \pt L \gg 1 \tand \{g_L^*\}(u)\gec 1 
 \pn \implies \text{either}\CAS{L/4<H\le L,\ S\sim R \sim He^{-H^2}.\\ H>L,\ R \sim He^{-H^2}.} }
To bound $S$ in the latter case, define $\e,\ti\e>0$ and $\ka\in\R$ by  
\EQ{
 \log(1/S) = \e HL, \pq \log(1/R) = H^2+\ka, \pq \de=\frac{L}{C_1\ti\e H}.} 
Then \eqref{bd S} implies $0<\ti\e \le \e$. Injecting the above into \eqref{H-L bd} yields 
\EQ{
 (H-L)^2 \pt\le \BR{1-\frac{L}{C_1\ti\e H}}(-\e HL+H^2+\ka)
 \pr\le H^2-\BR{\frac{1}{C_1\ti\e}+\e}HL+\frac{L^2}{C_1}+\BR{1-\frac{L}{C_1\ti\e H}}\ka,}
so 
\EQ{ \label{lbd ka}
 \ka \ge \BR{\frac{1}{C_1\ti\e}+\e-2}HL.}
In the case of \eqref{est R}, $\ka$ is upper bounded, hence the above estimate implies that $\ti\e$ is bounded away from $0$, and that $\e$ is bounded from above. Hence $\e\sim\ti\e\sim 1$, in other words,
\EQ{
 \log(1/S) \sim HL, \pq \de\sim L/H.} 
Then \eqref{ext f bd} is improved to
\EQ{
 \{\el_L^s\}_{(R,\I)}(u)=\ox_H.}

Next we investigate the higher part of $u$. By the change of variable $r=Re^{-t}$, we have in the case \eqref{est R},
\EQ{
 \{g_H^*\}(u) \le \int_0^\I \frac{e^{u^2}}{u^2}R^2e^{-2t}dt \sim \int_0^\I \frac{H^2}{u^2}e^{u^2-2H^2-2t}dt
	 \le \int_0^\I e^{u^2-2H^2-2t}dt.}
where the exponent can be rewritten as  
\EQ{
  u^2-2H^2-2t = -(u-2H)^2-2\BR{t-(u-H)^2}.}
Let $\fy(x)=u(Rx)-H$ for $|x|<1$. Then $\fy\in H^1_0(D)$ and $\|\na\fy\|_{L^2}^2=2\pi K_{(0,R)}(u)\le 2\pi$, for which the original Trudinger-Moser \eqref{oTM} reads 
\EQ{
 \int_0^\I e^{2(u-H)^2-2t}dt = \frac{1}{2\pi}\int_{D} e^{2\fy^2}dx \le C.}
Hence we have uniformly for $N>0$
\EQ{
 \{g_H^*\}(u1_{|u-2H|^2>N\log H}) \lec \int_{|u-2H|^2>N\log H}e^{-(u-2H)^2+2(u-H)^2-2t}dt
 \lec H^{-N}.}
In other words, we may restrict $u$ to $2H+O(\sqrt{N\log H})$: 
\EQ{
 \{g_H^*\}(u) = \{g_H^*\}(u1_{|u-2H|^2<N\log H}) + O(H^{-N}).}
In particular, $L<2H+1$ for large $L$. 
 
The above estimates are all concerned about the contribution in $|u|\ge L\gg 1$ and so the concentration part $|x|\ll 1$. 
To see the tail part, let $H\ge 1$ and 
\EQ{
 u(R_0)=1/H} 
at some $R_0>0$. Then 
\EQ{
 1=M(u) \ge M_{(0,R_0)}(u) \ge R_0^2/H^2}
implies $R_0\le H$, while \eqref{Schw} implies 
\EQ{
  \log(R_0/R) \ge (H-1/H)^2 \ge H^2-2,}
so in the case of \eqref{est R}, we have 
\EQ{ \label{tail R}
 R_0 \ge Re^{H^2-2} \sim H.}

Summarizing the above arguments, we have 
\begin{lem} \label{lem:redu}
Let $g^*_L$ be the critical function defined by \eqref{def f*}. 
For any $\e,N^*>0$, there exists $L_*\ge 1$ such that for any $L\ge L_*$ we have the following. For any $u\in X$ satisfying $\{g_L^*\}(u)\ge\e$, the parameters $S,\de,R,H$ defined by \eqref{def param} satisfy $0<S,R,H<\I$, $0<\de\le 2$, 
$2H+1>L$, $R \sim He^{-H^2}$, $\log(1/S) \sim HL$, $\de \sim L/H$, $M_{(0,R)}(u)=\ox_H$ and 
\EQ{ \label{int g localize}
 \{g^*_L\}(u)=\int_{|u-2H|^2<N\log H}g^*_L(u) rdr + O(H^{-N})}
uniformly for $0<N\le N^*$. If $u(r)=1/H$ then $r\sim H$. 
\end{lem}

For any $g:[0,\I)\to\R$ satisfying \eqref{cond f}, and $H>0$, let $\cS_H(g)$ be the supremum under the constraint $u(R)=H$, namely 
\EQ{ \label{def SH}
 \pt \cS_H(g):=\sup_{u\in X_H}\{g\}(u),
 \prq X_H:=\{u\in X\mid \exists R>0,\ K_{(0,R)}(u)\le 1, \pq K_{(R,\I)}(u)\le 1, \pq u(R)=H\}, }
then we have, in general, 
\EQ{
 \I>\cS(g)= \sup_{0<H<\I}\cS_H(g)>0.}
In Section \ref{sect:next exp}, we will prove

\begin{lem}
\label{lem asy SH 0}
Using the above notation, we have the asymptotic expansion
\EQ{ \label{asy SH 0}
 \cS_H(g^*_H)=\cS_\I+O(H^{-6}) \pq(H\to\I).}
\end{lem}
Taking this lemma granted, and using the above Lemma \ref{lem:redu}, we are able to prove the main Theorem \ref{thm:main}. 

\begin{proof}[Proof of Theorem \ref{thm:main}]
Let $g\in\cG$. For $H\to\I$, \eqref{asy SH 0} implies that there exists a sequence $u_H\in X_H$ such that $\{g^*_H\}(u_H)\to\cS_\I$. For any $\e\in(0,1)$, \eqref{cond f} implies that $g\ge(1-\e)g^*_H$ for sufficiently large $H>1$. Then 
\EQ{
 \cS(g) \ge (1-\e)\limsup_{H\to\I}\cS(g^*_H) \ge (1-\e)\limsup_{H\to\I}\cS_H(g^*_H)=(1-\e)\cS_\I.}
Since $\e>0$ is arbitrary, we deduce that $\cS(g)\ge\cS_\I$. 

Take any sequence $u_n\in X$ such that $\{g\}(u_n)\to\cS(g)$. 
If there are $R,\de>0$ such that $K_{(R,\I)}(u_n)\ge\de$ for large $n$, then $K_{(0,R)}(u_n)\le 2-\de$, while $u_n(R)$ is bounded by the radial Sobolev inequality. Hence the subcritical Trudinger-Moser implies that $g(u_n)r$ is uniformly integrable on $r\in(0,R)$. The radial Sobolev together with $g(s)=o(s)$ (as $s\to+0$) implies that $g(u_n)r$ is uniformly integrable on $r\in(R,\I)$. 
Hence $\{g\}(u_n)\to\{g\}(u)$ where $u_n\to u\in X$ is the weak limit. Thus $\cS(g)=\{g\}(u)$, so $g\in\cG_M$. 

Therefore, if $g\in\cG_N$ then for any maximizing sequence $\{u_n\}\subset X$ of $\cS(g)$, we have 
\EQ{ \label{conc K-R}
 \sup_{R>0}\limsup_{n\to\I}K_{(R,\I)}(u_n)=0,}
and so $u_n\to 0$ weakly in $H^1$ and locally uniformly on $r\in(0,\I)$. 
Hence $\{g\}(\min(u_n,L))$ converges to $0$ for any $L>0$. On the other hand, \eqref{cond f} implies for any $\e\in(0,1)$ that if $L>0$ is large enough then $g(s)\le(1+\e)g^*_L(s)$ for $s>L^2$. Hence 
\EQ{
 \cS(g)=\lim_{n\to\I}\{g\}(u_n) \le (1+\e)\limsup_{n\to\I}\{g^*_L\}(u_n).}
\eqref{conc K-R} also yields a sequence $R_n\to 0$ such that $K_{(0,R_n)}(u_n)\le 1$ and $K_{(R_n,\I)}(u_n)\le 1$. Let $H_n=u_n(R_n)$. If $\{H_n\}$ is bounded, then by the same argument as above, $\{g\}(u_n)\to\{g\}(u)=0$, a contradiction. Hence $H_n\to\I$, passing to a subsequence if necessary. Then 
\EQ{
 \cS_\I \le \cS(g) \le (1+\e)\limsup_{H\to\I}\cS_H(g^*_L).}
Applying Lemma \ref{lem:radSob} to the maximizer $u$ of $\cS_H(g^*_L)$ for $H\gg L+1$, we deduce that 
\EQ{
 \cS_H(g^*_L) = \{g^*_L\}(u) = \{g^*_H\}(u) + O(H^{-1}) \le \cS_H(g^*_H) + O(H^{-1}) \to \cS_\I,}
as $H\to\I$. Since $\e>0$ is arbitrary, we deduce that $\cS(g)=\cS_\I$. 

Next we prove the existence under the condition \eqref{ex by large}. For brevity, let 
\EQ{
 \s:=\sqrt{7\log H}.}
For $H\gg 1+L$, let $u\in X_H$ be the maximizer of $\cS_H(g^*_H)$. Then by \eqref{ex by large} and \eqref{int g localize}, 
\EQ{
 \{g\}(u) \pt\ge \int_0^\I \BR{g^*_L(u)\Br{1+a u^{-2p}+aO(u^{-2p-4})+O(u^{-6})}}rdr 
 \pr\ge \{g^*_H\}(u)\BR{1+ a\Br{2H+O(\s)}^{-2p} } + O(H^{-6}).}
Injecting the asymptotic expansion \eqref{asy SH 0} into $\{g^*_H\}(u)=\cS_H(g^*_H)$ yields 
\EQ{
 \{g\}(u) \ge\cS_\I\BR{1+a(2H)^{-2p}}+O(H^{-2p-1}\s+H^{-6}) > \cS_\I}
for large $H$, if $0<p<3$ and $a>0$, or if $p=3$ and $a\ge C_*>0$ is large enough compared with the $O(H^{-6})$ error. 

Similarly, in the case of condition \eqref{ex by small}, we have 
\EQ{
 \{g\}(u) \pt\ge \int_0^\I \BR{g^*_L(u)\Br{1+ bu^{-2q} +bO(u^{-2q-4}) + O(u^{-6})}+a u^{2+2p}}rdr
 \pr \ge \cS_\I+O(H^{-6}+H^{-2q})+a\int_0^\I u^{2+2p}rdr. }
Let $R_0>0$ such that $u(R_0)=H^{-1}$. Then $R_0\sim H$ and the last term is bigger than 
\EQ{
 a \int_0^{R_0}u^{2+2p} rdr \ge a R_0^2 H^{-2-2p}/2 \sim a H^{-2p}.}
Since $q>p$, this term dominates the error terms as $H\to\I$ provided that $a>0$ with $p<3$ or $a\ge C_*$ with $p=3$. 
Hence $\{g\}(u)>\cS_\I$ for large $H$, and so $g\in\cG_M$. 

Next we prove the non-existence under the condition \eqref{nex by large}.
Let $L_*\ge 1$ be given by Lemma \ref{lem:redu} for $\e:=\cS_\I$. 
Fix the parameters $p,q,a,b$ and assume \eqref{nex by large} for some $L\ge L_*$ to be taken large. 
Suppose that $u\in X$ and $\{g\}(u)\ge\cS_\I$, so that Lemma \ref{lem:redu} applies to $u$. 
If $L\ge L_*$ is large enough, then $H>L/2-1\gg 1$ and  
\EQ{
 \{g\}(u) \pt\le \int_0^\I g^*_L(u)\BR{1-a u^{-2p} + aO(u^{-2p-4}) + O(u^{-6})}rdr
 \pr\le \{g^*_H\}(u)\BR{1 - a\Br{2H+O(\s)}^{-2p} } + O(H^{-6})
 \pr\le \cS_H(g^*_H)\BR{1 - a\Br{2H+O(\s)}^{-2p} } + O(H^{-6}).}
Injecting the asymptotic expansion \eqref{asy SH 0} yields
\EQ{
 \{g\}(u) \le  \cS_\I\BR{1-a(2H)^{-2p}}+O(H^{-6}+H^{-2p-1}\s)<\cS_\I}
for large $H$, by the same comparison between $aH^{-2p}$ and the error terms as in the case \eqref{ex by large}. Hence $g\in\cG_N$ if $L$ is large enough. 

Similarly, in the case of \eqref{nex by small}, we have some constant $C>0$ such that
\EQ{
 \{g\}(u) \pt \le \cS_\I+O(H^{-6}+H^{-2q}) - a \int_0^\I u^{2+2p} rdr
 \pr \le \cS_\I - a C H^{-2p} + O(H^{-6}+H^{-2q}) <\cS_\I,}
for $L\ge L_*$ large enough, so $g\in\cG_N$. 
\end{proof}

\section{First expansion of the concentrating energy} \label{sect:loc int}
In this and the next sections, we consider asymptotic expansion of $\cS_H(g^*_H)$ as $H\to\I$, where $g^*_H$ was defined in \eqref{def f*}, and $\cS_H$ was introduced in \eqref{def SH}. 
We restrict $H$ to $H\gg 1$ so that $g^*_H(s)\sim s^{-1}\el_H^s$. 
Let $u\in X_H$. 
By the change of variables $r=Re^{-t}$ and $u(r)=H+v(t)$, we have 
\EQ{
 \pt \{g_H^*\}(u) = \frac{R^2}{H^2e^{-2H^2}}F_H(v), \pq F_H(v):=\int_0^\I \frac{e^{-(v-H)^2-2(t-v^2)-c_Eu^{-4}}}{(1+v/H)^2} dt,}
while $K_{(0,R)}(u)\le 1$ and $u(R)=H$ are respectively rewritten as 
\EQ{ \label{cond v}
 \int_0^\I |\dot v|^2 dt \le 1, \pq v(0)=0.}
On the other hand, for any increasing function $\ti v:[0,\I)\to[0,\I)$ satisfying the above conditions, let 
\EQ{
 \ti u(r):=\CAS{u(r) &(r\ge R)\\ \ti v(\log(R/r))+H &(0<r<R).}}
Then $\ti u$ is decreasing and $K_{(0,R)}(\ti u)\le 1$, hence by the same argument as for \eqref{int mass}, 
\EQ{
 M(\ti u)=M_{(0,R)}(\ti u)+M_{(R,\I)}(u)=M(u)+\ox_H \le 1+\ox_H,}
so there is $\la=1+\ox_H$ such that $\ti u(\la r)\in X$. Then 
\EQ{
 \frac{R^2}{H^2e^{-2H^2}}F_H(\ti v)=\{g^*_H\}(\ti u) \le \la^2\cS_H(g^*_H)=(1+\ox_H)\cS_H(g^*_H).} 
Hence, ignoring the error of $\ox_H$, $F_H(v)$ is maximized over all $v$ satisfying \eqref{cond v}, which is independent of $R$. 
Then we can maximize $R$ under the condition $K_{(R,\I)}(u)\le 1$, $M_{(R,\I)}(u)=1+\ox_H$ and $u(R)=H$, which is independent of $v$.  
It is equivalent to optimizing $\mu$ in Lemma \ref{lem:radSob}, hence the maximal $R$ is given by 
\EQ{
 R^2 = \frac{1+\ox_H}{\mu(2H^2)}=H^2e^{-2H^2}\ck\mu(2H^2)(1+\ox_H),}
where the function $\ck\mu:(0,\I)\to(0,\I)$ is defined by
\EQ{ \label{def muhat}
 \ck\mu(s):=\frac{2s^{-1}\el_1^s+s^2}{\mu(s)},}
so that $\ck\mu(s)\sim 1$ uniformly for $s>0$. 
Thus we obtain 
\EQ{ \label{SH 2 S0}
 \cS_H(g^*_H)\pt=\ck\mu(2H^2)\sS_0(H)+\ox_H, 
 \pr \sS_0(H):=\sup\Br{F_H(v) \Bigm| \int_0^\I|\dot v|^2dt=1,\ v(0)=0}.}
Putting $u=v+H$ and 
\EQ{
 \y(v,t)\pt:=u^2-2H^2-2t-2\log(1+v/H)-c_Eu^{-4}
 \pr=-2(t-v^2)-(v-H)^2-2\log(1+v/H)-c_Eu^{-4},} 
the maximized integral can be written in an exponential form
\EQ{
 F_H(v) = \int_0^\I e^{\y(v,t)}dt.}
The lower order term $c_Eu^{-4}$ in $\y$ has no role or even no effect in this section. 
It will affect the expansion only near the end of the next section. 
Actually, the role of $\log(1+v/H)=\log(u/H)$ is also small in this section.

Let $H\gg 1$ and let $v$ be a maximizer of $\sS_0(H)$. Then there is a Lagrange multiplier $\la\ge 0$ such that on $t>0$,
\EQ{ \label{eq v 0}
 \p_v e^{\y} = -\la\ddot v.}
Since $\y_v=\y_u\ge 2u-2/u>0$, the case $\la=0$ is precluded. 
Hence $\dot v$ is decreasing to $0$ (to be in $L^2(0,\I)$), so $v$ is increasing and concave.
Inner product of the above equation with $\dot v$, $1$, and $v$ respectively yields 
\EQ{ \label{int id}
 \pt \frac{\la a^2}{2} = 2\int_0^\I e^{\y} dt - e^{\y(0)}, \pq 
  \la a = \int_0^\I \p_v e^{\y}dt,
 \pq \la = \int_0^\I v\p_v e^{\y}dt,}
where the initial velocity denoted by 
\EQ{
 a:=\dot v(0)}
will be the central parameter in the following asymptotic analysis. 

Next consider the boundedness and integrability of $e^{\y}$. The Schwarz inequality \eqref{Schw} implies $t\ge v^2$. 
Moreover, the original Trudinger-Moser \eqref{oTM} on the disk yields
\EQ{ \label{TM bd}
 \int_0^\I e^{2(v^2-t)}dt \lec 1.}
Hence the $j$-th moment for $v$ around $H$, denoted by 
\EQ{
 I_j:=\int_0^\I(v-H)^j e^{\y}dt,}
 is bounded for each $j\ge 0$. On the other hand, the loss of compactness implies 
\EQ{ \label{conc}
 \liminf_{H\to\I}I_0>0.}

Using the moments, the integral identities \eqref{int id} are expanded as follows. 
Let $v':=v-H$. Since $\y_v=2u+O(1/u)$,  we have 
\EQ{
 \pt \y_v=2(2H+v')+O(H^{-1}),
 \pr v\y_v=2(H+v')(2H+v')+O(1)=4H^2+6Hv'+2(v')^2+O(1).}
Plugging them into \eqref{int id} yields
\EQ{
 \pt \la a^2  = 4I_0-\ox_H, 
 \pq \la a = 4HI_0+2I_1+O(H^{-1}),
 \pq \la = 4H^2I_0 + 6HI_1 + O(1).}
Dividing the second and third identities by the first, we obtain 
\EQ{
 \frac1a = H+\frac{I_1}{2I_0}+O(H^{-1}), \pq \frac{1}{a^2}=H^2+\frac{3I_1}{2I_0}H+O(1).}
Hence $|I_1|\lec 1/H$ and 
\EQ{ \label{ex a la}
 a = \frac{1}{H}+O(H^{-3}), \pq \la = 4H^2I_0 + O(1).}

\subsection{Exponential behavior}
Let $\te$ be the phase in the equation \eqref{eq v 0}, namely 
\EQ{ \label{def te}
 \te:=\y + \log\frac{\y_v}{\la}.}
Then the Euler-Lagrange equation is rewritten as 
\EQ{ \label{eq v}
 -\ddot v = e^{\te} = \frac{(e^{\y})_v}{\la} \lec \frac{u}{H^2}e^{2(v^2-t)-(v-H)^2},}
where $\la\sim H^2$ by \eqref{ex a la} was used. Since $v^2\le t$, we obtain 
\EQ{
 0<-\ddot v \lec H^{-1}. }
The phase $\te$ satisfies 
\EQ{ \label{te deri}
 \te_v= 2u+O(1/u), \pq \dot\te=\te_v\dot v-2.}
Hence, integrating $\dddot v=\ddot v(\te_v\dot v-2)$ from $t=\I$ yields
\EQ{ \label{eq w}
 \ddot v = -2\dot v - \int_t^\I \te_v \dot v \ddot v dt.}
For the last integrand, the bound on $|\ddot v|$ in \eqref{eq v} yields 
\EQ{
 0<-\te_v \ddot v \lec \frac{u^2}{H^2} e^{2(v^2-t)-(v-H)^2}.}
Hence there is an absolute constant $B_0>0$ such that as soon as $v(t)\ge H+B_0$ then by Trudinger-Moser \eqref{TM bd}
\EQ{
 \int_t^\I |\te_v \ddot v| dt \le 1.}
Suppose that $v(t_0)=H+B_0$ for some $t_0>0$. 
Then using the above together with the monotonicity of $\dot v$ in \eqref{eq w}, we deduce 
\EQ{
 t_0<t \implies -\ddot v \ge \dot v \pt\implies \dot v \le \dot v(t_0)e^{-t+t_0} 
 \pr\implies v \le v(t_0)+\dot v(t_0) \le H+B_0+a.}
Thus we obtain a priori bound for some absolute constant $B_1>0$
\EQ{ \label{upb v}
 0<\forall t<\I, \pq v(t)\le \lim_{t\to\I}v(t)=:v(\I) \le H+B_1.}

Let $T_*>0$ such that 
\EQ{
 v(T_*)=\frac23 H.}
Then $v^2\le t$ and \eqref{TM bd} applied to \eqref{eq v} imply, due to the factor $e^{-(v-H)^2}$, 
\EQ{ \label{asy T*}
 0<t\le T_* \implies \ddot v=\ox_H, \pq \dot v=a-\ox_H, \pq v=at-\ox_H,}
together with $T_*=\frac{2H}{3a}+\ox_H = \frac{2H^2}{3}+O(1)$, using \eqref{ex a la}.  

For any $k\in(0,1)$, there exists a unique $\dot T_k>0$ such that 
\EQ{
 \dot v(\dot T_k)=ka,}
since $\dot v(t)$ is strictly decreasing to $0$.
Using $\dot\te=(2u+O(H^{-1}))\dot v-2$, $a=1/H+O(H^{-3})$, and $\frac23H\le v\le H+B_1$ for $t>T_*$, we obtain 
\EQ{
 \CAS{T_*<t<\dot T_{2/3} \implies 2/9+O(H^{-2})<\dot\te<2+O(H^{-1}), \\
 t>\dot T_{1/3} \implies -2/3+O(H^{-1})>\dot\te>-2.}}

Thus $\ddot v$ is exponentially increasing on $[T_*,\dot T_{2/3}]$ and exponentially decreasing on $[\dot T_{1/3},\I)$. Integrating $\ddot v \sim \pm\dot\te\ddot v$ in $t$ from $T_*$ or $\I$, and using \eqref{asy T*} for $t<T_*$, we deduce that 
\EQ{
 \CAS{0<t<\dot T_{2/3} \implies -\ddot v \sim a-\dot v+\ox_H  \sim at-v+\ox_H,
  \\ \dot T_{1/3}<t<\I \implies -\ddot v \sim \dot v \sim v(\I)-v.}}
In particular, $-\ddot v(\dot T_{2/3})\sim H^{-1}\sim-\ddot v(\dot T_{1/3})$. Also, the upper bound $v\lec H$ by \eqref{upb v} implies $\dot T_{2/3}\lec H^2$. 

For $\dot T_{2/3}<t<\dot T_{1/3}$, as long as $-\ddot v\sim H^{-1}$, we have
\EQ{ \label{te on the top}
 \ddot\te=\ddot v\te_v + \dot v^2\te_{vv} \sim -1,}
which means that $\dot\te$ is decaying and $\te$ is concave. 
Hence by continuity (and using $-\ddot v\lec H^{-1}$), $-\ddot v=e^{\te}\sim H^{-1}$ is extended to $\dot T_{2/3} \le t \le \dot T_{1/3}$,  as well as \eqref{te on the top}, which implies $|\dot T_{2/3}-\dot T_{1/3}|\sim 1$. 

For brevity, denote 
\EQ{
 T_a:=\dot T_{1/2}, \pq t_a:=t-T_a.}
Gathering the above estimates, we obtain   
\EQ{ \label{exp loc}
 \pt 0<t<T_a \implies -\ddot v \sim a-\dot v+\ox_H \sim at-v + \ox_H = a\ox_{|t_a|}, 
 \pr T_a<t<\I \implies -\ddot v \sim \dot v \sim v(\I)-v =a\ox_{|t_a|}.}
where the error term $\ox_H$ is absorbed by $\ox_{|t_a|}$ on the right side of $t<T_a$, using that $T_a\sim H^2$. Note however that $\ox_H$ can not be ignored for the equivalence on the left, since the exponential behavior becomes degenerate, i.e. $|\dot\te|\ll 1$, near $t=0$. 

Using the above behavior of $\dot v$ and $T_a\lec H^2$ in the kinetic energy, we have 
\EQ{ \label{KEC 0}
 1=\int_0^\I\dot v^2dt\pt=\int_0^{T_a}a^2(1+\ox_{|t_a|})^2dt+\int_{T_a}^\I a^2\ox_{|t_a|}dt
 =a^2(T_a+O(1)),}
hence 
\EQ{
 T_a = \frac{1}{a^2}+O(1).}
Then using the behavior \eqref{exp loc} of $v$, as well as \eqref{ex a la}, we obtain 
\EQ{ \label{exp va}
  v(T_a)=\frac{1}{a}+O(a)=H+O(H^{-1}), \pq v(\I)=H+O(H^{-1}).}

\subsection{Soliton approximation}
Next we derive an approximate shape of $v$ around the transition $t=T_a$, using the equation \eqref{eq w}. Using \eqref{exp va} and \eqref{exp loc}, we have 
\EQ{ \label{tev 0}
 \te_v=2u+O(1/u)=2(H+v(T_a))+O(a\LR{t_a})=\frac 4a+O(a\LR{t_a}),}
where the growing factor (in $t_a$) is harmless when combined with the exponential decay of $\ddot v$.  
Then the equation \eqref{eq w} can be expanded as  
\EQ{ \label{eq w 0}
 \ddot v=\frac2a\dot v(\dot v-a)+\cR, \pq \cR(t):=\int_\I^t \ddot v\dot v(\te_v-4/a)dt.}
Using that $\cR(0)=\ddot v(0)=\ox_H$, the remainder can be written also 
\EQ{ \label{eq R}
 \cR(t)=\int_0^t \ddot v\dot v(\te_v-4/a)dt+\ox_H.}
Hence by \eqref{tev 0} and the exponential localization of $\ddot v$, see \eqref{exp loc}, we obtain 
\EQ{ \label{bd R}
 \cR = a^3\ox_{|t_a|}.}
The solution to the ODE for $\dot v$ without the remainder $\cR$, namely
\EQ{
 \dot w = \frac2a w(w-a), \pq w(T_a)=a/2}
is explicitly given by
\EQ{ \label{def w0}
 w=\frac a2w_0, \pq w_0(t):=1-\tanh t_a.}
Let $w_R$ be the remainder for $\dot v$ defined by 
\EQ{ \dot v=\frac{a}{2}(w_0+w_R),}
then it satisfies 
\EQ{ \label{eq wR}
 \dot w_R = -(2\tanh t_a)w_R + w_R^2 + \frac 2a\cR, \pq w_R(T_a)=0,}
which can be put in a Duhamel form 
\EQ{ \label{Duh wR}
 w_R = \int_0^{t_a}\frac{\cosh^2s}{\cosh^2t}\Br{w_R^2+\frac 2a\cR}(s+T_a)ds.}
Since $\cR/a=a^2\ox_{|t_a|}$ by \eqref{bd R}, a bootstrapping argument for this integral equation of $w_R$, using the exponential decay of $\sech^2t$, yields
\EQ{
 |w_R| \le a^2\ox_{|t_a|}.}
Injecting this and \eqref{bd R} into \eqref{eq wR} yields the same bound on $\dot w_R$. Thus we obtain 
\EQ{ \label{est wR}
 |w_R|+|\dot w_R| \le a^2\ox_{|t_a|}}
and so an expansion
\EQ{ \label{ex vt}
 \pt \dot v = \frac a2(w_0 + w_R)= \frac a2(1-\tanh t_a) + a^3\ox_{|t_a|},
 \pr \ddot v = \frac a2\dot w_0 + a^3\ox_{|t_a|} = -\frac a2 \sech^2 t_a + a^3\ox_{|t_a|},}
namely the soliton approximation. 
Moreover, integrating it from $t=0$ and from $t=\I$ yields respectively
\EQ{ \label{ex v}
 v =\CAS{ at-\frac a2\log(1+e^{2t_a}) + a^3\ox_{-t_a},
  \\v(\I)-\frac a2\log(1+e^{-2t_a}) + a^3\ox_{t_a}. } }
Note that we can not retain the exponential decay on the other side beyond $t=T_a$ in each integration. 
Comparing the two expressions (say at $t=T_a$) yields
\EQ{ \label{vI}
 v(\I) = aT_a + O(a^3).}
Let $v_\pm$ denote the two primitives of $w_0$ used above, namely
\EQ{ \label{def vpm}
 \pt v_- :=2t-\log(1+e^{2t_a})=2T_a-\log(e^{-2t_a}+1)
 \pn=\CAS{2t+\ox_{-t_a}, \\ 2T_a+\ox_{t_a},}
 \pr v_+ :=-\log(1+e^{-2t_a})
  =2t_a-\log(e^{2t_a}+1) = \CAS{\ox_{t_a}, \\2t_a+\ox_{-t_a},} }
satisfying $\dot v_\pm=w_0$, $v_-(T_a)=2T_a-\log 2$, and $v_+(T_a)=-\log 2$.  

\subsection{Expanding integral conditions}
Next we check a couple of global conditions on the approximation function obtained above. 
Firstly, the {\it kinetic energy condition}, which was already used in \eqref{KEC 0}, implies 
\EQ{ \label{ke id}
 1\pt=\int_0^\I\dot v^2 dt = \frac{a^2}{4}\int_0^\I w_0^2dt + O(a^4)
 \pr=\frac{a^2}{4}[2v_++w_0]_0^\I + O(a^4) = \frac{a^2}{4}(4T_a-2)+O(a^4),}
where $\dot w_0=w_0^2-2w_0=w_0^2-2\dot v_+$ was used for integration. 
Thus we obtain 
\EQ{ \label{T by a}
 T_a= \frac{1}{a^2}+\frac12+O(a^2).}

Secondly, {\it the initial acceleration condition}, which was already used in \eqref{eq R}, implies 
\EQ{ \label{ini acc}
 \ox_H = \ddot v(0) = \cR(0) = -\int_0^\I \ddot v\dot v(\te_v-4/a)dt.}
For the last term, using \eqref{ex v}, \eqref{vI} and \eqref{exp va}, we have
\EQ{ \label{tev 01}
 \te_v=2u-\frac{1}{u}+O(a^3)=2H+2aT_a+av_+-\frac{a}{2}+O(a^3\LR{t_a}).}
Plugging this as well as \eqref{ex v} into the above yields 
\EQ{
 O(a^5) \pt= -\int_0^\I \ddot v\dot v (2H+2aT_a-a/2-4/a) dt- \int_0^\I \frac{a^3}{4}\dot w_0 w_0 v_+ dt
 \pr= a^2(H+aT_a-a/4-2/a) - \frac{3a^3}{4}, }
where \eqref{int zwv} was used to compute the last integral. Thus we obtain 
\EQ{ \label{H+aT}
 H+aT_a = \frac2a + a + O(a^3),}
and by \eqref{T by a} and \eqref{vI}
\EQ{ \label{H by a}
 \pt H = \frac{1}{a} + \frac{a}{2} + O(a^3) = aT_a + O(a^3) = v(\I)+O(a^3), 
 \pr a = \frac{1}{H} + \frac{1}{2H^3} + O(H^{-5}).}

The coincidence of $H$, $aT_a$ and $v(\I)$ to the order $O(a^3)$ (instead of $O(a)=O(H^{-1})$) appears mysterious in our computation. 

\subsection{The main expansion}
Finally, to expand $\int_0^\I e^{2\y}dt$, we consider the main part of phase difference from the soliton approximation:
\EQ{
 \x\pt:=u^2-2H^2-2t+2\log(2\cosh t_a)=u^2-2H^2-2T_a - 2v_+.}
Let $v_R$ be a primitive of $w_R$ defined by 
\EQ{ \label{def vR}
  v=\frac{a}{2}(v_- + v_R) = aT_a + \frac{a}{2}(v_+ + v_R).}
Then \eqref{ex v} implies $v_R=a^2\ox_{-t_a}$. Using $a(H+aT_a)=2+a^2+O(a^4)$ from \eqref{H+aT}, 
\EQ{ \label{exp xi 0}
 \x\pt=\BR{H+aT_a+\frac{a}{2}(v_++v_R)}^2-2H^2-2T_a-2v_+
 \pr=(H+aT_a)^2 - 2H^2 - 2T_a + a^2v_+ + 2v_R + \frac{a^2}{4}v_+^2 + O(a^4).}
To expand the constant part, define $A,\hat H,\hat T$ by
\EQ{ \label{AHT exp}
A=\frac{1}{a}+\frac{a}{2},\pq  H=A+a^3\hat H, 
 \pq T_a=\frac{A}{a}+a^2\hat T}
so that $|\hat H|+|\hat T|\lec 1$ by \eqref{H by a}. Then 
\EQ{ \label{xi canc}
 \pt (H+aT_a)^2 -2H^2-2T_a =\BR{2A+a^3(\hat H+\hat T)}^2 - 2(A+a^3\hat H)^2 - \frac{2A}{a}-2a^2\hat T
 \pr=A(2A-2/a) + a^2(4aA-2)\hat T +O(a^6)
 \pr= 1 + a^2(2\hat T+1/2) + 2a^4\hat T + O(a^6)}
Plugging this into \eqref{exp xi 0}, we obtain 
\EQ{
 \x = 1 + a^2(2\hat T+1/2 + v_+) + 2v_R + \frac{a^2}{4}v_+^2 + O(a^4).}
Hence 
\EQ{ \label{exp exi}
 (u/H)^2e^{\y}\pt=e^{\x-2\log(2\cosh t_a)-c_Eu^{-2}} = \frac{e^{\x+O(a^4)}}{4}\sech^2 t_a
 \pr=\frac{e}{4}|\dot w_0|\BR{1+a^2(2\hat T+1/2+v_+ +v_+^2/4)+2v_R+O(a^4)}.}
Using \eqref{ex v} and \eqref{H by a}, we have 
\EQ{ \label{1/u}
 \frac{2H}{u}  \pt= \frac{2H}{2H+\tfrac a2 v_+ + O(a^3)} 
 =1-\frac{a^2}{4}v_+ + O(a^4\LR{t_a}^2),}
Combining this with \eqref{exp exi} yields  
\EQ{ \label{exp I0}
 e^{\y}= \frac{e}{16}|\dot w_0|\BR{1+\frac{a^2}{2}(1+v_++v_+^2/2)+2(a^2\hat T+v_R)} +a^4\ox_{|t_a|}.}
$a^2\hat T+v_R$ has cancellation coming from the kinetic energy condition
\EQ{
 1\pt=\int_0^\I \dot v^2 dt=\frac{a^2}{4}\int_0^\I(w_0^2+2w_0w_R+w_R^2)dt
 \pr=\frac{a^2}{4}\BR{4T_a-2+\ox_H +\int_0^\I(2w_0w_R+w_R^2)dt}.}
Rewriting it using partial integration with $\dot v_R=w_R$  and \eqref{est wR} 
\EQ{ \label{TvR}
 \hat T=a^{-2}\BR{T_a-\frac12-\frac1{a^2}} \pt=-\int_0^\I\frac{2w_0w_R+w_R^2}{4a^2}dt+\ox_H
 \pr=\frac{1}{2a^2}\int_0^\I \dot w_0v_Rdt+O(a^2),}
we obtain 
\EQ{
 \int_0^\I|\dot w_0|(a^2\hat T+v_R)dt = O(a^4).}
The fact that this expression appears in the integral $\sS_0$, namely \eqref{exp I0}, is another mysterious cancellation in our computation. 

Using it in \eqref{exp I0}, as well as explicit integration \eqref{int zwv}, we obtain
\EQ{ \label{exp sS 0}
 \sS_0(H) =\int_0^\I e^{\y}dt \pt= \frac{e}{16}\int_0^\I|\dot w_0|\BR{1+\frac{a^2}{2}(1+v_++v_+^2/2)}dt + O(a^4)
 \pr=\frac{e}{8}(1+a^2/2)+O(a^4).}
On the other hand, the expansion \eqref{exp mu} for the radial Sobolev can be rewritten for $\ck\mu$ defined in \eqref{def muhat} as 
\EQ{ \label{exp cmu}
 \ck\mu(2H^2)\pt=8e^{1-2\ga}\BR{1-\tfrac12H^{-2}-\tfrac18H^{-4}+O(H^{-6})}
 \pr=8e^{1-2\ga}\BR{1-\tfrac12a^2+\tfrac38a^4+O(a^6)},}
where \eqref{H by a} was used. 
Combining the above two expansions in \eqref{SH 2 S0}, we finally obtain an expansion of the constrained maximum: 
\EQ{ 
 \cS_H(g^*_H) = \ck\mu(2H^2)\int_0^\I e^{\y}dt + \ox_H = e^{2-2\ga}+O(H^{-4})=\cS_\I+O(H^{-4}),}
where $O(H^{-2})$ is absent by cancellation between the expansion of $\sS_0(H)$ and that of the radial Sobolev. This is another mysterious cancellation in our computation, which shows that the growth order $u^{-2}e^{u^2}$ is very critical, at least more than $e^{u^2}$ in the original Trudinger-Moser on the disk. 
Thus we have obtained the second order expansion in \eqref{asy SH 0}, but we need further to expand the next term. 

\section{The next expansion} \label{sect:next exp} 
In this section, we extend the expansion \eqref{exp sS 0} to the next order, namely Lemma \ref{lem asy SH 0}, which is needed to determine the existence for the critical function $s^{-1}\el_L^s$. 
The next term is obtained using the linearized equation around the soliton $\dot w_0$. 

\subsection{Linearized approximation} \label{ss:linearize}
In the Duhamel form \eqref{Duh wR}, first we have from \eqref{tev 01} and \eqref{H+aT}, 
\EQ{ \label{exp tev 1}
 \te_v = \frac 4a + av_+ + \frac{3a}{2} + O(a^3\LR{t_a}).}
Then $\cR$ is expanded as
\EQ{
 \frac8{a^3}\cR\pt=\int^t_\I \frac{8}{a^3}\dot v\ddot v(\te_v-4/a)dt = \int^t_0 \frac{8}{a^3}\dot v\ddot v(\te_v-4/a)dt+\ox_H
 \pn=\cR_0 + a^2\ox_{|t_a|},}
with the approximation, using the explicit integration by \eqref{prim zwv}, 
\EQ{
 \cR_0:=\int^t_\I w_0\dot w_0(2v_++3)dt = v_+(w_0^2-4)+w_0(w_0-2) = \ox_{|t_a|}.}
Then we extract the main part from \eqref{Duh wR} using \eqref{est wR}
\EQ{ \label{def w1}
  w_R = \frac{a^2}{4}w_1 +  a^4\ox_{|t_a|}, 
 \pq w_1 := \int_{T_a}^{t}\frac{\dot w_0(t)}{\dot w_0(s)}\cR_0(s)ds.}
In other words, $w_1$ is the solution to the linearized equation for the main remainder
\EQ{
 \dot w_1 + (2\tanh t_0)w_1 = \cR_0, \pq w_1(t_a)=0.}
The above integral is computed using $\dot w_0=w_0(w_0-2)$ and $(1/w_0)_t=2/w_0-1$, 
\EQ{ \label{w1 id}
 w_1 = \dot w_0\int_{T_a}^t[2v_+ + v_+(2/w_0-1)+1]dt = \dot w_0(2\nu_0+v_+/w_0) = \ox_{|t_a|},}
where a primitive of $v_+$ is introduced:
\EQ{ \label{def nu0}
 \nu_0 \pt:=\int_{T_a}^t v_+dt-v_+(T_a)/2=-\int_{0}^{t_a}\log(1+e^{-2s})ds + \tfrac12\log 2
 \pr=\CAS{-c_0 + \tfrac12\log 2 + \ox_{t_a},
  \\ t_a^2 + c_0 + \tfrac12\log 2 + \ox_{-t_a},} }
with the constant $c_0$ defined by
\EQ{
 c_0:=\int_{0}^\I\log(1+e^{-2t})dt=\sum_{k=1}^\I \int_0^\I \frac{(-1)^{k-1}}{ke^{2tk}}dt = \sum_{k=1}^\I\frac{(-1)^{k-1}}{2k^2}=\frac{\z(2)}{4}.}
Thus the next order expansion of $\dot v$ is given by 
\EQ{ \label{exp vt 1}
 \dot v = \frac a2w_0 + \frac{a^3}{8} w_1 + a^5 \ox_{|t_a|},}
with \eqref{w1 id}. Then using the equation \eqref{eq wR}, we also obtain 
\EQ{ \label{exp vtt 1}
 \ddot v = \frac a2\dot w_0 + \frac{a^3}{8}\dot w_1 + a^5 \ox_{|t_a|}.}
For the expansion of $v$, we need to integrate $w_1$
\EQ{
 v_1 \pt:=\int_\I^t w_1 dt = \int_\I^t[2\dot w_0\nu_0+(w_0-2)v_+]dt
 \pr=[2w_0\nu_0-v_+^2/2-2\nu_0]_\I^t=2w_0\nu_0-2(\nu_0-\nu_0(\I))-v_+^2/2 = \ox_{t_a},}
where $\nu_0(\I):=\lim_{t\to\I}\nu_0(\I)=c_0-\frac12\log 2$. 
Using \eqref{def vpm} and \eqref{def nu0}, we have
\EQ{ \label{v1 at 0}
 v_1 = 2\nu_0+2\nu_0(\I)-(2t_a)^2/2+\ox_{-t_a}=2\log 2+\ox_{-t_a}.}
Thus integrating the expansion of $\dot v$, we obtain 
\EQ{ \label{exp v 1}
 v \pt= v(\I) + \frac a2v_+ + \frac{a^3}{8} v_1 + a^5 \ox_{t_a}
 = \frac a2v_- + \frac{a^3}{8}(v_1-2\log 2) + a^5 \ox_{-t_a}.}
In particular,
\EQ{ \label{vI 1}
 v(\I) = aT_a - \frac{a^3}{4}\log 2 + O(a^5).}

\subsection{Expanding integral conditions}
The kinetic energy condition \eqref{ke id} is further expanded by \eqref{exp vt 1}
\EQ{ \label{exp ke 1}
 1\pt=\int_0^\I\dot v^2 dt = \frac{a^2}{4}\int_0^\I w_0^2dt + \frac{a^4}{8}\int_0^\I w_0w_1dt
 + O(a^6),}
where we can compute the next order integral using \eqref{w1 id} and $\dot w_0=w_0(w_0-2)$
\EQ{ \label{int w0w1}
 \int_0^\I w_0w_1dt \pt= \int_0^\I [2w_0\dot w_0\nu_0+(w_0^2-2w_0)v_+]dt
 \pr=[w_0^2\nu_0-v_+^2]_0^\I = -4c_0-2\log2 +\ox_H.}
Thus \eqref{exp ke 1} becomes 
\EQ{
 1 = \frac{a^2}{4}(4T_a-2) - \frac{a^4}{4}(2c_0+\log 2) + O(a^6),}
which yields the next expansion of the transition time 
\EQ{ \label{Ta 1}
 T_a=\frac{1}{a^2}+\frac 12 + \frac{a^2}{4}(2c_0+\log 2) + O(a^4),}
or $\hat T=\frac12c_0+\frac14\log 2+O(a^2)$. Injecting this into \eqref{vI 1} yields
\EQ{ \label{vI 1-2}
 v(\I)=\frac{1}{a}+\frac a2 + \frac{c_0}{2}a^3  + O(a^5).}
 
For the initial acceleration condition \eqref{ini acc}, we first expand
\EQ{ \label{exp tev 2}
 \pt\te_v-\frac 4a=2u-\frac{1}{u}+\frac{2}{u^3}-\frac 4a+O(a^5)
 \pr=2H + 2v(\I) -\frac 4a + a v_+ + \frac{a^3}{4} v_1
 - \frac{1}{2H}\BR{1-\frac{a^2}{4}v_+} + \frac{a^3}{4}+O(a^5\LR{t_a}^2)
 \pr=2H -\frac 2a + a(1+v_+) + \frac{a^3}{4}(4c_0 + v_1 + 1)
  \prQ- \BR{\frac a2-\frac{a^3}{4}}\BR{1-\frac{a^2}{4}v_+} + O(a^5\LR{t_a}^2)
 \pr=\frac a2(3+2v_+) + \frac{a^3}{8}(8c_0 + 16\hat H +4  +  v_+ + 2 v_1) +O(a^5\LR{t_a}^2),}
using \eqref{exp v 1}, \eqref{1/u}, \eqref{vI 1-2}, \eqref{H by a} and \eqref{AHT exp}. 
Then the remainder is expanded as
\EQ{
 \cR=\int_\I^t \ddot v\dot v(\te_v-4/a)dt=\frac{a^3}{8}\cR_0+\frac{a^5}{32}\cR_1+a^7\ox_{|t_a|},} 
with
\EQ{ \label{def R1}
 \cR_1:= \int_\I^t \dot w_0 w_0(8c_0+16\hat H+4+v_++2v_1) + (\dot w_0 w_1 + w_0\dot w_1)(3+2v_+)dt.}
Since $\cR_0(0)=\ox_H$, the initial acceleration condition implies 
\EQ{ \label{R1=0}
 \cR_1(0) = O(a^2).}
The cubic part containing $w_1,\dot v_1$ is computed by 
\EQ{ \label{R1 new int}
 \pt\int_0^\I 2(\dot w_0 w_0 v_1 + \dot w_0 w_1 v_+  + w_0\dot w_1 v_+ )dt 
 \pr=[(w_0^2-4)v_1+2w_0w_1v_+]_0^\I + \int_0^\I (4-3w_0^2)w_1dt
 \pr=\int_0^\I (4-3w_0^2)(2\dot w_0\nu_0+v_+(w_0-2)) dt + \ox_H,}
where the part with $\nu_0$ is integrated by parts 
\EQ{
 \int_0^\I 2(4-3w_0^2)\dot w_0\nu_0dt = [2(4w_0-w_0^3)\nu_0]_0^\I-\int_0^\I 2w_0(4-w_0^2)v_+dt,}
so
\EQ{
 \eqref{R1 new int}\pt=\int_0^\I v_+(w_0-2)[4-3w_0^2+2w_0(w_0+2)]dt + \ox_H
 \pr=\int_0^\I[v_+\dot w_0(4-w_0) + 4v_+(w_0-2)]dt + \ox_H
 \pn=5+16c_0+\ox_H,}
where the last integral can be computed by using \eqref{int wwv} and \eqref{int c0}. 
Thus \eqref{R1=0} becomes, using \eqref{int zwv}, 
\EQ{ \label{exp R1=0}
 O(a^2)\pt=(8c_0+16\hat H+4)\int_0^\I \dot w_0w_0dt + \int_0^\I \dot w_0w_0v_+ dt
 + 3[w_0w_1]_0^\I + \eqref{R1 new int}
 \pr=-2(8c_0+16\hat H+4)+3+5+16c_0+\ox_H=-32\hat H+\ox_H,}
and so $\hat H = O(a^2)$, which is another unexpected cancellation. Thus we obtain  
\EQ{ \label{H 1}
  H=\frac{1}{a}+\frac{a}{2}+O(a^5).}

\subsection{The main expansion}
To expand $\int_0^\I e^{\y}dt$, we first improve the expansion of $\x$ from \eqref{exp xi 0}, using \eqref{xi canc}, \eqref{Ta 1} and \eqref{H 1}
\EQ{ \label{exp xi 1}
 \x\pt=1 + a^2(2\hat T+1/2) + 2a^4\hat T + (a^2+a^4\hat T)v_+ 
 \prQ+ (2+a^2)v_R + \frac{a^2}{4}v_+^2 + \frac{a^2}{2}v_+v_R + O(a^6)
 \pr=1 + \frac{a^2}{2}(4\hat T+1-2\log 2+2v_++v_+^2/2+v_1) 
 \prq+ \frac{a^4}{4}\BR{8\hat T-2\log 2+(4\hat T-\log 2)v_+ + v_1 + \frac{v_1 v_+}{2} +2v_R^1} + O(a^6)
 \pr=1 + \frac{a^2}{2}(c_1 + 2v_+ + v_+^2/2 + v_1)
 \prq + \frac{a^4}{4}(4c_0+ 2c_0 v_+  + v_1 + v_1 v_+/2  + 8\hat T_1 + 2v_R^1) + O(a^6),}
where $\hat H$ is dropped by \eqref{H 1}, while $\hat T,v_R$ are expanded, respectively by using \eqref{AHT exp} and \eqref{Ta 1}, and by using \eqref{def vR} and \eqref{exp v 1}: 
\EQ{ \label{exp T vR}
 \hat T = \frac12 c_0 + \frac14\log 2 + a^2\hat T_1, \pq v_R = \frac{a^2}{4}(v_1-2\log 2+ a^2v_R^1), \pq v_R^1=\ox_{-t_a},}
which defines $\hat T_1$ and $v_R^1$, and $c_1$ is the constant defined by 
\EQ{ \label{def c1}
 c_1:=2c_0+1-\log 2 = \frac{\z(2)}{2}+1-\log 2.}

Next we improve \eqref{1/u}, using \eqref{exp v 1}, \eqref{vI 1-2} and \eqref{H 1},  
\EQ{
 \frac{2H}{u}\pt=\frac{2H}{2H + \tfrac a2 v_+ + \tfrac{a^3}{2}(c_0+v_1/4) + O(a^5)} 
 \pr=1 - \frac{a}{4H} v_+ - \frac{a^3}{4H}(c_0+v_1/4) + \frac{a^2}{16H^2} v_+^2 + O(a^6\LR{t_a}^3)
 \pr=1-\frac{a^2}{4}v_+ - \frac{a^4}{16}(4c_0-2v_++v_1-v_+^2) + O(a^6\LR{t_a}^3).}
The two expansions yield
\EQ{ \label{exp ey 1}
 \pt e^{\y}=\frac{(2H/u)^2}{16\cosh^2t_a}e^{\x-c_Eu^{-4}}
 \pr=\frac{e|\dot w_0|}{16}\Bigl[1-\frac{a^2}{2}v_+ - \frac{a^4}{8}(4c_0-2v_++v_1-3v_+^2/2)\Bigr]
 \times\BR{1-c_E(2H)^{-4}}
 \prq\times\Bigl[1+ \frac{a^2}{2}(c_1 + 2v_+ + v_+^2/2 + v_1)+ \frac{a^4}{8}(c_1 + 2v_+ + v_+^2/2 + v_1)^2
 \prQ\pQ+ \frac{a^4}{4}(4c_0+ 2c_0 v_+  + v_1 + v_1 v_+/2  + 8\hat T_1 + 2v_R^1)\Bigr] + a^6\ox_{|t_a|}.}
This is the first place where $c_E$ makes a difference. Hence we have
\EQ{ \label{exp ey 2}
 e^{\y}=\frac{e|\dot w_0|}{16}\BR{1+\frac{a^2}{2}f_1+\frac{a^4}{4}f_2+a^4(2\hat T_1+v_R^1/2)}+a^6\ox_{|t_a|},}
where $f_1 = c_1 + v_+ + v_+^2/2 + v_1$ and 
\EQ{
 \pt f_2 = -2c_0+v_+-v_1/2+3v_+^2/4 - c_E/4 + 4c_0+2c_0v_+ + v_1 + v_1 v_+/2 
 \prQ + (c_1+2v_++v_+^2/2+v_1) ^2/2 - v_+(c_1+2v_++v_+^2/2+v_1) 
 \prq= 2c_0 + \tfrac12 c_1^2 - c_E/4 + (1+2c_0+c_1)v_+
 \pn+ \BR{\tfrac 34 + \tfrac{c_1}2} v_+^2 + \tfrac 12 v_+^3 + \tfrac 18 v_+^4
 \prQ+ \BR{\tfrac 12+c_1} v_1 + \tfrac 32 v_1v_+ + \tfrac 12 v_1v_+^2 + \tfrac 12 v_1^2
 \prq= c_4 + c_5 v_+ + c_6 v_+^2 + \tfrac12v_+^3 + \tfrac18v_+^4
 +c_7 v_1 + \tfrac32v_1v_+ + \tfrac12 v_1v_+^2 + \tfrac 12v_1^2,}
with the constants $c_4$--$c_7$ defined by 
\EQ{
 \pt c_4:=2c_0 + \tfrac 12c_1^2 -c_E/4, 
 \pq c_5:=1+2c_0+c_1, 
 \pq c_6:=\tfrac 34 + \tfrac 12c_1, 
 \pq c_7:=\tfrac 12+c_1.}
For the $O(a^4)$ term, the polynomial of $v_+$ in $f_2$ is easily integrated by \eqref{int zwv}
\EQ{
 \pt\int_0^\I|\dot w_0|(c_4 + c_5 v_+ + c_6 v_+^2 + \tfrac12v_+^3 + \tfrac18v_+^4)dt
 =2c_4-2c_5+4c_6+\ox_H.}
The next term is computed using \eqref{v1 at 0} and \eqref{int w0w1}
\EQ{
 \pt\int_0^\I|\dot w_0|v_1dt=[w_0v_1]_\I^0 + \int_0^\I w_0w_1dt
 \pr=2v_1(0)-4c_0-2\log2+\ox_H=2\log2-4c_0+\ox_H.}
Hence, using \eqref{def c1}, we obtain 
\EQ{
 \int_0^\I|\dot w_0|c_7v_1dt = 2c_7(1-c_1)+\ox_H.}

The integral with $2\hat T_1+v_R^1/2$ is computed by improving \eqref{TvR}, using \eqref{exp vt 1},
\EQ{
 \hat T \pt=-\int_0^\I\frac{2w_0w_R+w_R^2}{4a^2}dt+\ox_H
 \pn=\int_0^\I\BR{\frac{1}{2a^2} \dot w_0v_R - \frac{a^2}{64} w_1^2}dt + O(a^4),}
so, injecting \eqref{exp T vR},
\EQ{
 \frac12 c_0 + \frac14\log 2 + a^2\hat T_1 = \int_0^\I \BR{\frac{\dot w_0}{8}(v_1-2\log 2+ a^2v_R^1)
- \frac{a^2}{64} w_1^2}dt + O(a^4).}
The term on the right of $O(1)$ is computed using \eqref{int w0w1}
\EQ{ \label{int z0v1}
 \int_0^\I \dot w_0(v_1-2\log 2)dt=-\int_0^\I w_0w_1dt + \ox_H = 4c_0+2\log2+\ox_H,}
thereby we confirm that $O(1)$ terms match. Thus we obtain 
\EQ{ \label{TvR 1}
 \int_0^\I |\dot w_0|(2\hat T_1+v_R^1/2)dt = -\frac{1}{16}\int_0^\I w_1^2dt + O(a^2).}

Thus we are lead to integrate, ignoring the $\ox_H$ errors, 
\EQ{ \label{hard int}
 -\frac18\int_\R \dot w_0(3v_1v_+ +  v_1v_+^2 + v_1^2)dt-\frac{1}{16}\int_\R w_1^2 dt.}
For the left integral, we apply partial integration to $w_0(w_0-2)=\p_t(w_0-2)$: 
\EQ{ \label{hint 1}
 \pt-\int_\R \dot w_0(3v_1v_+ +  v_1v_+^2 + v_1^2)dt
 \pr=\int_\R (w_0-2)w_1(3v_+ + v_+^2 + 2v_1)dt + \int_\R (w_0-2)w_0(3v_1+2v_1v_+)dt,}
where the last integral is equal to, by the same partial integration, 
\EQ{
 -\int_\R(w_0-2)w_1(3+2v_+)dt - \int_\R(w_0-2)2w_0v_1dt,}
and the latter term is equal to 
\EQ{
 2\int_\R(w_0-2)w_1dt.}
For the other term in \eqref{hard int}, we have 
\EQ{
 w_1^2 = (2w_0\nu_0+v_+)(w_0-2)w_1.}
Thus we obtain
\EQ{
 \eqref{hard int}=\frac18\int_\R (w_0-2)w_1(v_+/2 + v_+^2 + 2v_1 -1-w_0\nu_0)dt.}
Next we use for partial integration
\EQ{ \label{ibp w1}
 (w_0-2)w_1=(w_0-2)[2\dot w_0\nu_0+v_+(w_0-2)]
 =\p_t[(w_0-2)^2\nu_0].}
Then we have 
\EQ{
 \eqref{hard int}\pt=-\frac18\int_\R(w_0-2)^2\nu_0(w_0/2+ w_0v_+ + 2w_1 - \dot w_0\nu_0)dt
 \pn- \frac{c_\I}{2},}
where 
\EQ{
 c_\I := \nu_0(\I) = \frac{1-c_1}{2}.}
Integrate by parts the term with $w_1$ using \eqref{ibp w1}. Then
\EQ{
 \pt-\int_\R (w_0-2)^2\nu_0 w_1dt = -\int_\R(w_0-2)w_1(w_0-2)\nu_0dt
 \pr=\int_\R (w_0-2)^2\nu_0[\dot w_0\nu_0+(w_0-2)v_+]dt+8c_\I^2}
Then 
\EQ{
 \eqref{hard int}\pt=-\frac18\int_\R\BR{(w_0-2)^2(w_0/2-w_0v_++4v_+)\nu_0 -3 (w_0-2)^2\dot w_0\nu_0^2}dt
 \prq-\frac{c_\I}{2}+2c_\I^2,}
where the last part of integral is computed using $\p_t(w_0-2)^3=3(w_0-2)^2\dot w_0$
\EQ{
 \frac38\int_\R(w_0-2)^2\dot w_0\nu_0^2dt
 =-c_\I^2-\frac14\int_\R(w_0-2)^3\nu_0v_+dt.}
Hence 
\EQ{
 \eqref{hard int}\pt=-\frac18\int_\R\BR{(w_0-2)^2(w_0/2+w_0v_+)\nu_0}dt
 \pn-\frac{c_\I}{2}+c_\I^2,}
then using $\p_t[(w_0-2)^2v_+/2]=(w_0-2)^2(w_0/2+w_0v_+)$, 
\EQ{
 \eqref{hard int}=\frac{1}{16}\int_\R(w_0-2)^2v_+^2dt -\frac{c_\I}{2}+c_\I^2 .}
For the last integral, we need a classical formula for the zeta function. 
\begin{lem} \label{zeta int}
For any $j\in\N$ and $k\in\C$ with $\re k>0$, we have 
\EQ{ \label{zeta formu}
 \int_\R (2-w_0)^j|v_+|^k dt 
 \pn= 2^{j-1}\Ga(k+1)\Br{\z(k+1)-\sum_{1\le n<j}n^{-k-1}}.}
\end{lem}
\begin{proof}
Changing variables by $s=-t_a$ and $x=\log(1+e^{2s})$ yields 
\EQ{
 \int_\R(2-w_0)^j|v_+|^k dt = \int_\R (1-\tanh s)^j[\log(1+e^{2s})]^kds = \int_0^\I \frac{(2e^{-x})^{j-1}x^k}{e^x-1}dx,}
since 
\EQ{
 1-\tanh s=2e^{-x}, \pq dx=2e^{2s-x}ds, \pq e^{2s}=e^x-1.} 
By the Taylor expansion, the above integral is equal to 
\EQ{
 2^{j-1}\sum_{n=1}^\I \int_0^\I x^k e^{(1-j-n)x}dx=2^{j-1}\sum_{n=j}^\I n^{-k-1}\Ga(k+1),}
which is equal to the right side of \eqref{zeta formu}. 
\end{proof}
Using the above lemma, we obtain 
\EQ{
 \eqref{hard int}\pt= \frac{\z(3)-1}{4} - \frac{c_\I}{2} + c_\I^2 .}

Adding the above computations, we conclude
\EQ{ \label{exp S0H 1}
 \sS_0(H)=\frac{e}{8}\BR{1+\frac{a^2}{2}+\frac{a^4}{8}c_8}+O(a^6),}
with 
\EQ{
 c_8\pt:=2c_4-2c_5+4c_6 + 2c_7(1-c_1) + \z(3)-1 - 2c_\I + 4c_\I^2 
 \pr= -c_E/2 + 1+\z(3)=-1.}
This is another unexpected cancellation, which is obtained by writing all the constants in terms of $c_1$ and $\z(3)$.  The constant $c_E$ was chosen to make $c_8=-1$.  
Combining it with \eqref{exp cmu}, we finally obtain 
\EQ{ \label{asy SH next}
 \cS_H(g_H^*)\pt=\ck\mu(2H^2)\sS_0(H)+\ox_H
 \pr=e^{2-2\ga}\BR{1+\tfrac12 a^2 + \tfrac{c_8}{8} a^4}\BR{1-\tfrac12a^2+\tfrac38a^4}+O(a^6)
 \pr=e^{2-2\ga}\BR{1+\frac{c_8+1}{8}a^4+O(a^6)}
 \pn=e^{2-2\ga}+O(H^{-6}),  } 
as claimed before in \eqref{asy SH 0}, which in particular proves Lemma \ref{lem asy SH 0}.

\section{Asymptotic expansion of the exponential radial Sobolev} \label{sect:exp Sob}
In this section, we prove the asymptotic formula for the radial exponential Sobolev inequality, Lemma \ref{lem:radSob}. First by rescaling, the optimal function $\mu$ for \eqref{sharp radS} is given by 
\EQ{
 \pt\mu(j)=\inf_{u\in X^1(j)}M_{(1,0)}(u), 
 \prq X^1(j):=\{u:[1,\I)\to\R \mid 2u(1)^2=j,\ K_{(1,\I)}(u)=1,\ M_{(1,\I)}(u)<\I\}.}
It is easy to see that this infimum $\mu(j)$ is attained for each $j>0$ by taking a sequence $\fy_n\in X^1(j)$ such that $M_{(1,\I)}(\fy_n)\searrow \mu(j)$, since \eqref{Schw} applied to $\fy_n$ implies that $\fy_n(r)\to\exists\fy(r)$ locally uniformly on $r\in(0,\I)$ (up to a subsequence), then $2\fy(1)^2=j$ and $M_{(1,\I)}(\fy)\le\mu(j)$, $K_{(1,\I)}(\fy)\le 1$ by weak convergence. If $K_{(1,\I)}(\fy)<1$, then consider
\EQ{
 \psi(t,r):=\fy(r+t(r-1))}
for $t,r>0$, which satisfies $2\psi(t,1)^2=j$ and 
\EQ{
 K_{(1,\I)}(\psi(t))=\int_1^\I|\fy'(r)|^2(r+t)dr,
 \pq M_{(1,\I)}(\psi(t))=\int_1^\I|\fy(r)|^2\frac{r+t}{(1+t)^2}dr.}
Hence $K_{(1,\I)}(\psi(t))$ is increasing in $t\in(0,\I)$ from $K_{(1,\I)}(\fy)$ to $\I$, while $M_{(1,\I)}(\psi(t))$ is decreasing, so we have $K_{(1,\I)}(\psi(t))=1$ and $M_{(1,\I)}(\psi(t))<\mu(j)$ for some $t>0$, which contradicts the definition of $\mu(j)$. Hence $\fy_n(|x|)\to\fy(|x|)$ strongly in $H^1(|x|>1)$ and $\fy$ is a minimizer for $\mu(j)$. 

Since $\fy$ is a constrained minimizer, there is a Lagrange multiplier $a\in\R$ such that for any $\psi\in C^\I(1,\I)$ with a compact support in $(1,\I)$, 
\EQ{
 \p_t|_{t=0}M_{(1,\I)}(\fy+t\psi)=a\p_t|_{t=0}K_{(1,\I)}(\fy+t\psi),}
i.e., $\fy(|x|)=-a\De\fy(|x|)$ in $\D'(|x|>1)$. The above argument implies that $a\le 0$, and obviously $a\not=0$ because of $\fy(1)>0$. Moreover, the elliptic regularity implies that the equation holds in the classical sense. 
Thus we deduce that for each $j>0$, there exist $\fy(j)\in X^1(j)$ and $\la(j)>0$ such that 
\EQ{
 \pt \mu(j)=M_{(1,\I)}(\fy(j)), \pq \De\fy(j)=\la(j)^2\fy(j) \pq(|x|>1).}
The ODE in $r$ has the unique solution with finite $M_{(1,\I)}$ and $K_{(1,\I)}$ for the boundary condition $\fy(1)=\sqrt{j/2}$, which is given in terms of the Bessel potential on $\R^2$: 
\EQ{
 G(x):=\F^{-1}(1+|\x|^2)^{-1}=\frac{1}{(2\pi)^2}\int_{\R^2}\frac{e^{ix\x}}{1+|\x|^2}d\x=\ox_{|x|}.}
Indeed we have 
\EQ{ \label{eq fyjmu}
 \pt \fy=K_{(1,\I)}(G(\la x))^{-1/2}G(\la x)=K_{(\la,\I)}(G)^{-1/2}G(\la x),
 \pr j=2\fy(1)^2=2K_{(\la,\I)}(G)^{-1}G(\la)^2,
 \pr \mu(j)=M_{(1,\I)}(\fy)=\la^{-2}K_{(\la,\I)}(G)^{-1}M_{(\la,\I)}(G).}
The second equation gives the relation between $j$ and $\la$, then plugging it into the third one yields a formula for $\mu(j)$. 
$M_{(\la,\I)}(G)$ and $K_{(\la,\I)}(G)$ can be written in terms of $G(\la),G_r(\la)$, using the energy and the Pokhozaev identity for $G$, which solves 
\EQ{
 r>0\implies \p_r(r G_r) = rG.}
Multiplying it with $G$ and integration yields 
\EQ{
 [GrG_r]_\la^\I-K_{(\la,\I)}(G) = M_{(\la,\I)}(G),}
and the multiplier $2rG_r$ yields 
\EQ{
 [(rG_r)^2]_\la^\I=[r^2G^2]_\la^\I-2M_{(\la,\I)}(G).}
Thus we obtain 
\EQ{
 \pt M_{(\la,\I)}(G)=\frac{\la^2(G_r(\la)^2-G(\la)^2)}{2},
 \pr K_{(\la,\I)}(G)=-\la G(\la)G_r(\la)+\frac{\la^2(G(\la)^2-G_r(\la)^2)}{2}.}
Plugging this into \eqref{eq fyjmu} and putting $\Te(\la):=-\la G_r(\la)/G(\la)>0$ yields
\EQ{ \label{eq jmu}
 \pt\frac{2}{j}=\frac{\la^2}{2}-\la\frac{G_r(\la)}{G(\la)}-\frac{\la^2}{2}\frac{G_r(\la)^2}{G(\la)^2}=\frac{\la^2}{2}+\Te(\la)-\frac{\Te(\la)^2}{2},
 \pr\mu(j)=\frac{M_{(\la,\I)}(G)}{\la^2K_{(\la,\I)}(G)}=\frac{j(\Te(\la)^2-\la^2)}{4\la^2 }.}

Asymptotic formulas for $G$ can be derived from the Laplace transform 
\EQ{
 \pt G=(1-\De)^{-1}\de=\int_0^\I e^{t\De-t}\de dt=\int_0^\I\frac{1}{4\pi t}e^{-\frac{r^2}{4t}-t}dt.}
Changing the variable $t=s^2$ and $\ta=(s-r/(2s))^2$ yields, noting that $(0,\I)\ni s\mapsto \ta\in(0,\I)$ is a 2-to-1 mapping except $s=\sqrt{r/2}$, 
\EQ{
 \pt G=\frac{e^{-r}}{2\pi}\int_0^\I e^{-(s-r/(2s))^2}\frac{ds}{s}
 =\frac{e^{-r}}{2\pi}\int_0^\I \frac{e^{-\ta}}{\sqrt{\ta(\ta+2r)}}d\ta,
 \pr G_r=-\frac{e^{-r}}{2\pi}\int_0^\I \frac{e^{-\ta}}{\sqrt{\ta(\ta+2r)}}\left[1+\frac{1}{2r}-\frac{\ta}{2r(\ta+2r)}\right]d\ta.}
Hence by the dominated convergence
\EQ{ \label{est GH}
 \pt \lim_{r\to\I}\sqrt{r}e^r G = \frac{1}{2\sqrt{2\pi}},
 \pq 0<r+\frac{1}{2}-\Te \lec r^{-1},}
so \eqref{eq jmu} implies that $\la$ should be bounded as $j\to\I$. 

On the other hand, \eqref{eq jmu} implies that $\la\to\I$ as $j\to+0$, and then combining \eqref{est GH} and \eqref{eq jmu}, it is easy to obtain 
\EQ{ \label{radSob j0}
 \la=\frac{4}{j}+O(1),\pq \mu(j)=\frac{1}{16} j^2+O(j^3)\pq(j\to+0).}

If $\la\to\exists\la_0>0$ along some sequence $j\to\I$, then \eqref{eq jmu} implies that 
\EQ{
 \Te(\la_0)^2-2\Te(\la_0)-\la_0^2=0 \iff \Te(\la_0)=1+\sqrt{\la_0^2+1}>\la_0+1,}
contradicting $\Te(r)<r+1/2$ in \eqref{est GH}. 
Therefore $\la\to 0$ as $j\to\I$. 

For asymptotic behavior as $r\to+0$, we have 
\EQ{
 2\pi e^rG \pt=\int_0^1\frac{1}{\sqrt{\ta(\ta+2r)}}+\frac{e^{-\ta}-1}{\sqrt{\ta(\ta+2r)}}d\ta+\int_1^\I \frac{e^{-\ta}}{\sqrt{\ta(\ta+2r)}}d\ta,
 \pr=[\log(\sqrt{\ta(\ta+2r)}+\ta+r)]_0^1
 \prQ +\int_0^1\frac{e^{-\ta}-1}{\ta}+O(r/(\ta+r))d\ta+\int_1^\I\frac{1+O(r/\ta)}{\ta}e^{-\ta}d\ta
 \pr=-\log r+\log 2-\ga+O(r\log r),}
where $\ga$ denotes Euler's constant, coming from the formula
\EQ{
 \pt\int_0^1\frac{1-e^{-\ta}}{\ta}d\ta - \int_1^\I \frac{e^{-\ta}}{\ta}d\ta=\ga.}
For $rG_r$, going back to the integral formula in $t$ and integrating by parts, we obtain
\EQ{
 -2\pi rG_r=\int_0^\I\frac{r^2}{4 t^2}e^{-\frac{r^2}{4t}-t}dt = \int_0^\I e^{-\frac{r^2}{4t}-t}dt,}
so 
\EQ{
 2\pi rG_r+1 = \int_0^\I (1-e^{-\frac{r^2}{4t}})e^{-t}dt
 \pt=\int_0^1\int_0^\I\frac{r^2}{4t}e^{-\frac{r^2\te}{4t}-t}dtd\te
 \pr=r^2\int_0^1 \pi G(r\sqrt\te)d\te.}
Plugging the above formula of $G$, we obtain 
\EQ{
 2\pi rG_r \pt=-1+\frac{r^2}{2}\int_0^1(-\log(r\sqrt\te)+\log 2-\ga+O(r\log r))d\te
 \pr= -1+\frac{r^2}{2}[\log(1/r)+1/2+\log 2-\ga]+O(r^3\log r),}
and so
\EQ{ \label{la 2 Te}
 \Te(r) &= \frac{1}{\log(1/r)+\log 2-\ga+O(r\log r)}.}
Hence as $j\to\I$, we have $\Te(\la)\sim 1/j$ and $\la=\ox_j$. 
Denoting for brevity 
\EQ{
 \chi:=\frac{1}{1+\sqrt{1-4/j}} = \frac12+\frac{2\chi^2}{j}=
\frac 12 + \frac{1}{2j} + O(j^{-2}),}
we obtain from \eqref{eq jmu} and \eqref{la 2 Te}
\EQ{
 \pt \frac{1}{\Te} = \frac{1}{1-\sqrt{1-4/j+\la^2}} = \frac{j}{4\chi} + \ox_j = \frac j2 - \chi + \ox_j,
 \pr \log(2/\la) = \ga+1/\Te+O(\la\log\la) = \ga + \frac j2 - \chi +\ox_j,}
and thus 
 \EQ{ \label{asy mu}
 \frac{1}{\mu(j)} \pt= \frac{4}{j(\Te^2/\la^2-1)} =\frac{4e^{-2\ga-j+2\chi}}{j}(j-2\chi)^2(1+\ox_j)
 \pr=4je^{-2\ga-j+1}\BR{1-\frac{1}{j}-\frac{4\chi^2}{j^2}}^2e^{4\chi^2/j}(1+\ox_j)
 \pr=4je^{-2\ga-j+1}\BR{1-j^{-1}-\frac1{2}j^{-2}-\frac5{6}j^{-3}-\frac{43}{24}j^{-4}-\frac{529}{120}j^{-5}+O(j^{-6})}.}

\section{The case of the disk} \label{sect:disk}
In this section, we prove Theorem \ref{thm:disk} on the disk $D$. 
The same argument as in the $\R^2$ case works, but it is simpler on $D$ by the following reasons. 
First, the decoupling into the central and the tail parts is almost trivial in this case. The tail part is treated by the elementary optimization of the Schwarz inequality \eqref{Schw}, instead of the exponential radial Sobolev inequality.
 
Now we start the proof of Theorem \ref{thm:disk}, in close comparison with the $\R^2$ case. 
It is convenient to introduce the following critical nonlinearity on the disk:
\EQ{ \label{asy f*}
 f^*_L(s):=\el_L^s e^{-s^{-1}-c_D's^{-2}} = \CAS{e^s(1-s^{-1}-c_Ds^{-2}) &(s\ge L^2),\\ 0 &(s\le L^2),}}
with a cut-off parameter $L>0$, where
\EQ{
 c_D':=c_D+1/2=2+2\z(3).}

First by the rearrangement, the existence of maximizer is the same for
\EQ{
 \pt \cS^D(g)=\sup_{u\in X^D}\{g\}(u^2), 
 \prQ X^D:=\{u:(0,1]\to[0,\I) \mid u'\le 0,\ K(u)\le 2,\ u(1)=0\}.}
Let $L\gg 1$ such that $f^*_L(s)\sim \el_L^s$, $u\in X^D$ with $\{f^*_L\}(u)\gec 1$, and define $S,\de,H,R$ in the same way by \eqref{def param}. 
By the same reasoning as on $\R^2$, we may assume, for large $L>1$, that $0<S,R<1$, $u(S)=L$, $u(R)=H$ and $K_{(R,1)}(u)=1\ge K_{(0,R)}(u)$. 
By the same change of variable $r=Re^{-t}$ and $u=v+H$, we have
\EQ{ \label{int el in t}
 \{f^*_H\}(u)=R^2e^{2H^2}\int_0^\I e^{2v^2-2t-(v-H)^2-u^{-2}-c_D'u^{-4}}dt.}
Since Schwarz \eqref{Schw} implies $Re^{H^2}\le 1$, while the original Trudinger-Moser \eqref{oTM} implies 
$\int_0^\I e^{2v^2-2t}dt \le \cS^D(\el_0^{s})<\I$, 
we deduce that $R \sim e^{-H^2}$. 
If $L\ge 4H$, then \eqref{small H int} implies 
\EQ{
 \{\el_L^s\}(u)  \le e^{-\frac{L^2}{16}}R^2 \le e^{-\frac{L^2}{16}},}
so $L<4H$ for large $L>1$. 
Since \eqref{Schw} implies $Se^{L^2/\de}\le 1$, the argument for \eqref{lbd ka} works with $C_1=1$. Since $\ka$ is bounded due to $R\sim e^{-H^2}$, we deduce in the same way as there that $\log(1/S)\sim HL$, $\de\sim L/H$, and $\{\el_L^s\}_{(R,\I)}(u) = \ox_H$. Moreover, 
\EQ{ \label{int el localize}
 \{f^*_H\}(u) = \int_{|u-2H|^2<N\log H}f^*_H(u^2)rdr + O(H^{-N})}
uniformly for $N>0$. Hence $2H+1>L$ for large $L$. If $u(R_0)=1/H$ at some $R_0\in(0,1)$, then by the same argument as for \eqref{tail R}, we have 
\EQ{ \label{tail R D}
 R_0 \ge R e^{H^2-2} \gec 1.}

To prove Theorem \ref{thm:disk}, by the same reasoning as on $\R^2$, it suffices to prove that $\{g\}(u)>\cS^D_\I=e$ for some $u\in X^D$ for the existence part, and that $\{g\}(u)<\cS^D_\I$ for all $u\in X^D$ for the non-existence part, where 
\EQ{
 \cS^D_\I := \lim_{H\to\I}\cS^D_H(f^*_H),}
with the constrained maximization defined by
\EQ{
 \pt \cS^D_H(g):=\sup_{u\in X^D_H}\{g\}(u),
 \prq X^D_H:=\{u\in X^D \mid \exists R\in (0,1),\ K_{(0,R)}(u)\le 1,\ K_{(R,\I)}(u)\le 1,\ u(R)=H\}. }
Thus the proof is reduced to the expansion
\EQ{ \label{exp SD}
 \cS^D_H(f^*_H)=e + O(H^{-6}), \pq(H\to\I).}

Taking it granted, let us finish the proof of Theorem \ref{thm:disk}. 
For the existence under the condition \eqref{ebl D}, let $H\gg L+1$ and let $u\in X^D_H$ be a maximizer of $\cS^D_H(f^*_L)$. Denote $\s:=\sqrt{7\log H}$. Then, using the above expansion, as well as \eqref{asy f*}, we obtain 
\EQ{
 \{g\}(u) \pt\ge \int_0^1 f^*_L(u^2)[1+a u^{-2p}+O(u^{-6})]rdr
 \pr\ge \{f^*_L\}(u)\BR{1+a\Br{2H+O(\s)}^{-2p}}+O(H^{-6})
 \pr=\cS^D_\I[1+a(2H)^{-2p}]+O(H^{-6}+H^{-2p-1}\s)>\cS^D_\I,}
for large $H$. Hence $\cS^D(g)$ is attained. 
In the case of condition \eqref{ebs D}, we have 
\EQ{
 \{g\}(u) \pt\ge \int_0^1 \BR{f^*_L(u^2)\Br{1+bu^{-2q}+O(u^{-6})}+a u^{2p}}rdr
 \pr\ge \cS^D_\I + O(H^{-6}) + O(H^{-2q}) + a \int_0^1 u^{2p} rdr,}
where the last term is estimated from below by 
\EQ{
 a \int_0^1 u^{2p} rdr \ge a R_0^2 H^{-2p}/2 \sim a H^{-2p},}
which dominates the error terms as $H\to\I$. Hence $\{g\}(u)>\cS^D_\I$ for large $H$ and so $\cS^D(g)$ is attained. 

For the non-existence under the condition \eqref{nbl D}, let $L_*\ge 1$ be large enough to have the above estimates (up to \eqref{tail R D}) for all $u\in X^D$ satisfying $\{g\}(u)\ge\cS^D_\I$. Fix the parameters $p,q,a,b$ and assume \eqref{nbl D} for some $L\ge L_*$. 
Suppose that $u\in X^D$ and $\{g\}(u)\ge\cS^D_\I$, so that we can apply the above argument (up to \eqref{tail R D}) to $u$.  Then 
\EQ{
 \{g\}(u) \pt\le \int_0^1 f^*_L(u^2)[1-a u^{-2p}+O(u^{-6})]rdr
 \pr\le \{f^*_H\}(u)\BR{1-a\Br{2H+O(\s)}^{-2p}} +O(H^{-6})
 \pr\le \cS^D_H(f^*_H)\BR{1-a\Br{2H+O(\s)}^{-2p}} +O(H^{-6})
 \pr\le \cS^D_\I[1-a(2H)^{-2p}]+O(H^{-6}+H^{-2p-1}\s).}
If $L\ge L_*$ is large enough, then $H>L/4\gg 1$ and so $\{g\}(u)<\cS^D_\I$. Hence $\cS^D(g)$ is not attained. In the case of condition \eqref{nbs D}, we have some constant $C>0$ 
\EQ{
 \{g\}(u) \pt\le \cS^D_\I + O(H^{-6}+ H^{-2q}) - a\int_0^1 u^{2p}rdr
 \pr\le \cS^D_\I - C a  H^{-2p} + O(H^{-6}+ H^{-2q}) < \cS^D_\I}
for $L\ge L_*$ large enough, so $\cS^D(g)$ is not attained.  

It remains to prove the expansion \eqref{exp SD} for $\cS^D_H(f^*_H)$. 
This is done in the same way as for $\cS_H(g^*_H)$ in Section \ref{sect:loc int}. To avoid repeating the same computations, we will describe only but thoroughly the differences from the $\R^2$ case. 
We will see even more coincidence than expected between the two cases despite of some numerical differences. In particular, the approximation of $v$ in terms of $a$ by the soliton and the linearization remains exactly the same, including $T_a$ and $v(\I)$, even though the relation between $H$ and $a$ is different. This suggests that those asymptotic expansions may have some universal character for the critical inequalities of Trudinger-Moser type.

Let $u\in X^D$ be a maximizer for $\cS^D_H(f^*_H)$. 
First, the change of variables to the form of $\sS_0$ in \eqref{SH 2 S0} is immediate in this case: Obviously  we should maximize the radius $R$ at $u=H$ by spending the half kinetic energy in $R<r<1$, as there is no other factor which should be taken account of. Moreover, this maximization is simply given by 
\EQ{
 u(r)=\log(1/r)[\log(1/R)]^{-1/2},}
as is well known and an easy consequence of \eqref{Schw}. Thus we obtain $R=e^{-H^2}$ and 
\EQ{
 \cS^D_H(f^*_H)=\sS_0(H):=\sup\Br{2\int_0^\I e^{\y}dt \Bigm| \int_0^\I|\dot v|^2dt=1,\ v(0)=0},}
with 
\EQ{
 \y \pt:= u^2-2H^2-2t -u^{-2} -c_D' u^{-4}
 \pr= -2(t-v^2)-(v-H)^2 -u^{-2} - c_D' u^{-4}. }

Henceforth the same symbols are used as in Section \ref{sect:loc int}, even if the expression may be slightly different between the two cases on $\R^2$ and on $D$ (such as $\y$ and $\sS_0(H)$ above). 
With this new $\y$, the same argument after \eqref{SH 2 S0} works with no difference until \eqref{tev 01}, which should be modified to 
\EQ{ \label{tev 01 D}
 \te_v=2u + \frac1u +O(a^3)=2H+2aT_a+av_+ + \frac{a}{2} +O(a^3\LR{t_a}),}
namely the sign of the last term before the error. 
This change results in modifying \eqref{H+aT} as 
\EQ{ \label{H+aT D}
 H+aT_a = \frac{2}{a} + \frac a2 + O(a^3),}
and so \eqref{H by a} as 
\EQ{ \label{H 2 a D}
 \pt H=\frac{1}{a}+O(a^3),  \pq a=\frac{1}{H}+O(H^{-5}), 
 \pr aT_a=\frac{1}{a}+\frac a2+O(a^3)=v(\I)+O(a^3),}
thus the relation between $H$ and $a$ is modified, while the last formula for $aT_a$ and $v(\I)$ remains the same. 

Then the modified \eqref{H+aT D} affects \eqref{exp xi 0} in the $v_+$ term as 
\EQ{ \label{exp xi 0 D}
 \x=(H+aT_a)^2 - 2H^2 - 2T_a + \frac{a^2}{2}v_+ + 2v_R + \frac{a^2}{4}v_+^2 + O(a^4).}
Then the definition of $\hat H$ in \eqref{AHT exp} is modified to 
\EQ{ \label{def hat H D}
 H=1/a+a^3\hat H,} 
while \eqref{xi canc} is replaced with
\EQ{ \label{xi canc D}
 \pt (H+aT_a)^2 - 2H^2 - 2T_a 
 \pr= \BR{2/a+a/2+a^3(\hat H+\hat T)}^2 -2(1/a+a^3\hat H)^2 - 2/a^2-1-2a^2\hat T
 \pr= 1 + a^2(2\hat T+1/4) + a^4(\hat H+\hat T) + O(a^6),}
and plugging this into above, we obtain 
\EQ{
 \x= 1 + a^2(2\hat T + 1/4 + v_+/2 + v_+^2/4) + 2v_R +O(a^4).}
Using this and 
\EQ{ \label{exp u-2 0}
 u^{-2}=(2H+O(a\LR{t_a}))^{-2}=\frac{a^2}{4}+O(a^4\LR{t_a}^2)}
instead of \eqref{1/u}, we replace \eqref{exp I0} with 
\EQ{
 e^{\y} =\frac{e}{4}|\dot w_0|\BR{1 + \frac{a^2}{2}(v_+ + v_+^2/2) + 2(a^2\hat T+v_R)}+a^4\ox_{|t_a|}.}
This is the first place where the lower order term $u^{-2}$ in $\y$ affects the expansion (we would get an extra $a^2/4$ without it). 
After using the same cancellation of $a^2\hat T+v_R$ as before, we obtain, in place of \eqref{exp sS 0},  
\EQ{
 \cS^D_H(f^*_H)\pt=2\int_0^\I e^{\y}dt
 =\frac{e}{2}\int_0^\I |\dot w_0|\BR{1+\frac{a^2}{2}(v_+ + v_+^2/2)}dt + O(a^4)
 \pr=e +O(a^4).}

Now we proceed to the next order, following Section \ref{sect:next exp}. 
First, from \eqref{tev 01 D} and \eqref{H+aT D} we see that $\te_v$ satisfies the same expansion as \eqref{exp tev 1}, and so Section \ref{ss:linearize} works the same. 
The first difference in that section appears in \eqref{exp tev 2}, where we use  
\EQ{ \label{exp 1/u}
 1/u = \BR{\frac2a+\frac{a}{2}(v_++1)+O(a^3)}^{-1} = \frac a2 - \frac{a^3}{8}(v_++1) + O(a^5\LR{t_a}^2).}
Then we modify \eqref{exp tev 2}, noting \eqref{def hat H D} as well,  
\EQ{ \label{exp tev 2 D}
 \pt\te_v-\frac 4a=2u+\frac{1}{u}+\frac{2}{u^3}-\frac 4a+O(a^5)
 \pr=2H -\frac 2a + a(1+v_+) + \frac{a^3}{4}(4c_0 + v_1 + 1)
  \pn+ \BR{\frac a2 - \frac{a^3}{8}(v_++1)} + O(a^5\LR{t_a}^2)
 \pr=\frac a2(3+2v_+) + \frac{a^3}{8}(8c_0 + 16\hat H +1  -  v_+  + 2 v_1) +O(a^5\LR{t_a}^2).}
Hence $\cR_1$ should be modified from \eqref{def R1} to 
\EQ{
  \cR_1:= \int_\I^t \dot w_0 w_0(8c_0+16\hat H+1-v_++2v_1) + (\dot w_0 w_1 + w_0\dot w_1)(3+2v_+)dt,}
which changes \eqref{exp R1=0} to 
\EQ{
  O(a^2)\pt=(8c_0+16\hat H+1)\int_0^\I \dot w_0w_0dt - \int_0^\I \dot w_0w_0v_+ dt
 + \eqref{R1 new int}
 \pr=-2(8c_0+16\hat H+1)-3+5+16c_0+\ox_H=-32\hat H+\ox_H,}
leading to the same cancellation $\hat H=O(a^2)$. Thus \eqref{H 1} is replaced with 
\EQ{
 H = \frac 1a + O(a^5).}
Next, \eqref{exp xi 1} should be modified, according to \eqref{H+aT D} and \eqref{xi canc D}, 
\EQ{
  \x\pt=1 + a^2(2\hat T+1/4) + a^4\hat T + (a^2/2+a^4\hat T)v_+ 
 \prQ+ (2+a^2/2)v_R + \frac{a^2}{4}v_+^2 + \frac{a^2}{2}v_+v_R + O(a^6)
 \pr=1 + \frac{a^2}{2}(2c_0+1/2-\log 2 + v_+ + v_+^2/2 + v_1)
 \prq + \frac{a^4}{4}(2c_0+ 2c_0 v_+  + v_1/2 + v_1 v_+/2  + 8\hat T_1 + 2v_R^1) + O(a^6).}
Using this and \eqref{exp 1/u}, we obtain, in place of \eqref{exp ey 1}--\eqref{exp ey 2}, 
\EQ{
 e^\y \pt= \frac{e^{\x-u^{-2}-c_D'u^{-4}}}{4\cosh^2t_a}
 \pr=\frac{e|\dot w_0|}{4}\exp\Biggl[\frac{a^2}{2}(2c_0 -\log 2 + v_+ + v_+^2/2 + v_1)
 \prq + \frac{a^4}{4}\Br{(2c_0+1/2+v_1/2)(1+ v_+)  + 8\hat T_1 + 2v_R^1-c_D'/4}\Biggr]+a^6\ox_{|t_a|}
 \pr=\frac{e|\dot w_0|}{4}\BR{1+\frac{a^2}{2}f_1+\frac{a^4}{4}f_2+a^4(2\hat T_1+v_R^1/2)} + a^6\ox_{|t_a|}.}
where $f_1$ and $f_2$ are modified as follows: 
\EQ{
 \pt f_1 = 2c_0 -\log 2  + v_+ + v_+^2/2 + v_1, 
 \pr f_2 = f_1^2/2 + (2c_0+1/2+v_1/2)(1+ v_+)
 \prq= c_4 + c_5 v_+ + c_6 v_+^2 + \tfrac 12v_+^3 + \tfrac 18 v_+^4 + c_7v_1 + \tfrac 32v_1v_+ + \tfrac12v_1v_+^2 + \tfrac12v_1^2,}
where $c_4$--$c_7$ are modified to (while $c_0,c_1$ are the same as before)
\EQ{
 c_4 = \tfrac12 c_1^2 + \log 2 -c_D'/4, \pq c_5=c_1+2c_0-\tfrac12, \pq c_6=\tfrac12c_1, \pq c_7=c_1-\tfrac12.}
Hence the computation of integrals of $O(a^4)$ remains the same, which leads to
\EQ{
 2\int_0^\I e^\y dt = e\BR{1+\frac{a^4}{8}c_8}+O(a^6),}
with $c_8$ defined as before (with the same $c_\I=(1-c_1)/2$ as before)
\EQ{
 c_8 \pt= 2c_4-2c_5+4c_6 + 2c_7(1-c_1) + \z(3)-1 - 2c_\I + 4c_\I^2 
 \pr= -c_D'/2 + 1+\z(3) =0.}
Intriguingly, the part $1+\z(3)$ is the same as in the case of $\R^2$, even though $c_4$--$c_7$ are  different! Of course, the constant $c_D=c_D'-1/2$ was chosen to make $c_8=0$. 
Using \eqref{H 2 a D}, we finally obtain \eqref{exp SD}.

\appendix 

\section{Some explicit integrals} \label{app:exp int}
For polynomials in $w_0$ and $v_+$ multiplied with $\dot w_0$ (see \eqref{def w0} and \eqref{def vpm} for definition of $w_0,v_+$), we have the following formulas. 
For any $k\in\N$ and any polynomial $\fy(v)$, 
\EQ{
 \p_t[(w_0-2)^k\fy(v_+)]\pt=k(w_0-2)^{k-1}\dot w_0\fy(v_+)+(w_0-2)^kw_0\fy'(v_+)
 \pr=(w_0-2)^{k-1}\dot w_0(k+\p_v)\fy(v_+),}
using $\dot w=w_0(w_0-2)$. 
Hence with
\EQ{
 \psi(v):=(k+\p_v)^{-1}\fy(v)=-\sum_{j=0}^\I(-k)^{-j-1}\fy^{(j)}(v),}
we have 
\EQ{ \label{prim zwv}
 \dot w_0(w_0-2)^{k-1}\fy(v_+) = \p_t[(w_0-2)^k\psi(v_+)],}
and so, for any $j\in\N_0$, 
\EQ{ \label{int wwv}
 \int_0^\I \dot w_0(w_0-2)^{k-1}v_+^j dt = \int_0^\I w_0(w_0-2)^k v_+^jdt
 =2^k(-1)^{k+j}k^{-j-1}j! +\ox_H.}
In particular we have
\EQ{ \label{int zwv}
 \pt \int_0^\I \dot w_0v_+^j dt = 2(-1)^{j+1}j!+\ox_H,
 \pr \int_0^\I \dot w_0w_0v_+^j dt = 4(-1)^{j+1}(1-2^{-j-1})j!+\ox_H.}
Also we have 
\EQ{ \label{int c0}
 \int_0^\I (w_0-2)v_+dt = \BR{v_+^2/2-2\nu_\I}_0^\I = 4c_0 + \ox_H,}
which can be computed also by Lemma \ref{zeta int}. 

\section{Asymptotic expansions for $\sS_0(H)$} \label{app:asy exp}
In this section, the asymptotic expansions derived for $\sS_0(H)$  in Sections \ref{sect:loc int}--\ref{sect:next exp} are summarized for the reader's convenience. See \eqref{SH 2 S0} for the definition of $\sS_0(H)$. 
Let $v$ be a maximizer of $\cS_0(H)$. Then 
\EQ{
 \pt \dot v(0)=a=\frac1H+\frac{2}{H^3}+O(H^{-5}), \pq H=\frac 1a+\frac a2+a^3\hat H=\frac 1a+\frac a2+O(a^5), 
 \pr u=H+v, \pq v(0)=0,\pq v(\I)=\frac1a+\frac{a}2+\frac{a^3}{2}c_0+O(a^5), 
 \pr \dot v(T_a)=\frac a2,
 \pq T_a=\frac{1}{a^2}+\frac12+a^2\hat T=\frac{1}{a^2}+\frac12+\frac{a^2}{4}(2c_0+\log 2)+O(a^4),
 \pr v=\frac a2(v_-+v_R)=aT_a+\frac a2(v_++v_R)=\CAS{\frac a2v_-+\frac{a^3}{8}(v_1-2\log 2)+a^5\ox_{-t_a} \\ v(\I)+\frac a2v_+ + \frac{a^3}{8}v_1+a^5\ox_{t_a},}
 \pr \dot v=\frac a2(w_0+w_R)=\frac a2w_0+\frac{a^3}{8}w_1+a^5\ox_{|t_a|},
 \pq \ddot v=\frac a2\dot w_0+\frac{a^3}{8}\dot w_1+a^5\ox_{|t_a|},
}
with $t_a:=t-T_a$, $c_0:=\z(2)/4$, and
\EQ{
 \pt v_-=2t-\log(1+e^{2t_a}), \pq v_+=-\log(1+e^{-2t_a}), 
 \pr v_1=-2\nu_0 \tanh t_a-2c_0+\log 2-\frac 12v_+^2,
 \pr w_0=\dot v_\pm=1-\tanh t_a,
 \pq w_1=\dot v_1=(w_0-2)(2w_0\nu_0+v_+),
}
where $\nu_0$ can be defined by, using the dilogarithm function $\Li_2$, 
\EQ{
 \nu_0:=-\int_0^{t_a}\log(1+e^{-2s})ds+\frac12\log 2=\frac12\Li_2(-e^{-2t_a})-c_0+\frac12\log 2.}

In the case of disk in Section \ref{sect:disk}, the expansion of $a$ is changed to
\EQ{
 a=\frac1H+O(H^{-5}), \pq H=\frac1a+O(a^5),}
but all the other expansions in the above list remain the same (in terms of $a$).

\section{The exact Trudinger-Moser from $\R^2$ to the disk} \label{app:R2 2 D}
In this section, we derive the original Trudinger-Moser inequality \eqref{oTM} on the disk $D:=\{x\in\R^2\mid|x|<1\}$, using the Trudinger-Moser on $\R^2$ with the exact growth condition \eqref{TM cond B}. 
The opposite direction was given in \cite[Appendix B]{MS}, but relying on the exponential radial Sobolev inequality \eqref{exp radS}. 
Since \eqref{exp radS} follows immediately from the exact Trudinger-Moser on $\R^2$, one might regard the latter as the master inequality among those three. 
It would be interesting if one can relate them concerning the existence of maximizer. 

Let $u\in H^1_0(D)$ be a radial decreasing function with $K(u)\le 2$. 
Take $R\in(0,1)$ such that $K_{(0,R)}(u)\le 1$ and $K_{(R,1)}(u)\le 1$, and let $H:=u(R)$ and $\ell:=|\log R|$. Then the simple Schwarz estimate \eqref{Schw} implies $H^2\le\ell$ or equivalently $e^{2H^2}\le R^{-2}$. 
Let $v:=\max(u-H,0)\in H^1_0(RD)$ and $w:=u-v$. Then $K(v)=K_{(0,R)}(u)\le 1$ and $K(w)=K_{(R,\I)}(u)\le 1$, hence 
\EQ{
 \int_R^1e^{u^2}rdr \le \int_0^1 e^{w^2}rdr \lec 1,}
by the subcritical Trudinger-Moser, either on $\R^2$ or on $D$. In particular, the exact inequality is more than sufficient here.  
For $r<R$, we have 
\EQ{ \label{inner int}
  e^{u^2} = e^{2v^2-(v-H)^2}e^{2H^2} \le R^{-2} e^{2v^2}e^{-(v-H)^2}.}
Let 
\EQ{
 I:=\{r\in(0,R) \mid e^{(v-H)^2}\ge 1+v^2\}.}
Then 
\EQ{
 \int_Ie^{u^2}rdr \le R^{-2}\int_0^R \frac{e^{2v^2}}{(1+v^2)}rdr
 \lec R^{-2}\|v\|_{L^2(RD)}^2,}
by the exact Trudinger-Moser on $\R^2$ with \eqref{TM cond B}. This is bounded, since we have 
\EQ{
 \int_0^R v^2 rdr \le \int_0^R \log(R/r)rdr \sim R^2,}
applying \eqref{Schw} to $v$. 

It remains to estimate the integral on $J:=(0,R)\setminus I$.  
We may assume $H\ge 1$, since otherwise \eqref{Schw} is sufficient:
\EQ{
 0<r<R \pt\implies u(r) 
 \le H+\sqrt{\log(R/r)}
 \pn\implies e^{u(r)^2} \le e^{CH^2}(R/r)^{3/2}.} 
For $H\ge 1$, $1+v^2\ge e^{(v-H)^2}$ implies that $v\sim H$ on $J$. 
Let $\ti u:=\min(u,4H)\in H^1_0(D)$. Then $K(\ti u)\le K(u)\le 2$ and 
\EQ{ \label{int J}
 \int_J e^{u^2}rdr \lec \int_0^R \frac{H^2}{\ti u^2+1}e^{\ti u^2}rdr
 = \int_0^{RH} \frac{e^{\fy^2}}{\fy^2+1}rdr,}
where $\fy(x)\in H^1_0(HD)$ is radial decreasing defined by 
\EQ{
 \fy(r)=\CAS{\ti u(r/H) &(r<RH)\\ \frac{H}{\ell}\log(H/r) &(RH<r<H).}}
Then 
\EQ{
 K(\fy)\pt=K_{(0,RH)}(\fy)+K_{(RH,\I)}(\fy)
 \pr\le K_{(0,R)}(u) + \int_{RH}^H \frac{H^2}{\ell^2 r^2}rdr
 \le 1+H^2/\ell \le 2.}
Hence the exact Trudinger-Moser on $\R^2$ with \eqref{TM cond B} implies 
\EQ{
 \eqref{int J} \lec \|\fy\|_2^2 
 = H^2M_{(0,R)}(\ti u)+\int_{RH}^H\frac{H^2|\log(H/r)|^2}{\ell^2}rdr
 \lec H^4R^2 + \frac{H^4}{\ell^2} \lec 1,}
thus we obtain the desired bound.


\begin{thebibliography}{10}
\bibitem{ASM}C.~O.~Alves, M.~A.~S.~Souto and M.~Montenegro, {\it Existence of a ground state solution for a nonlinear scalar field equation with critical growth.} Calc. Var. Partial Differential Equations {\bf 43} (2012), no. 3-4, 537--554.
\bibitem{BN} H.~Br\'ezis and L.~Nirenberg, {\it Positive solutions of nonlinear elliptic equations involving critical Sobolev exponents.} Comm. Pure Appl. Math. {\bf 36} (1983), no. 4, 437--477.
\bibitem{CC} L.~Carleson and S.-Y.~A.~Chang, {\it On the existence of an extremal function for an inequality of J.~Moser.} Bull. Sc. Math. (2) {\bf 110} (1986), no. 2, 113--127.
\bibitem{FMR}
D.~G.~de Figueiredo, O.~H.~Miyagaki and B.~Ruf, {\it Elliptic equations in $R^2$ with nonlinearities in the critical growth range}, Calc. Var. Partial Differential Equations {\bf 3} (1995), 139--153. 
\bibitem{FR}
D.~G.~de Figueiredo and B.~Ruf, {\it Existence and Non-Existence of Radial Solutions for Elliptic Equations with Critical Exponent in $\R^2$}. Comm. Pure Appl. Math. {\bf 48} (1995), 639--655.
\bibitem{FOR}
D.~G.~de Figueiredo, J.~M.~do \'O and B.~Ruf, {\em On an inequality by N.~Trudinger and J.~Moser and related elliptic equation}, Comm. Pure Appl. Math. {\bf 55} (2002) 1--18. 
\bibitem{TMX} S.~Ibrahim, N.~Masmoudi and K.~Nakanishi, {\it Trudinger-Moser inequality on the whole plane with the exact growth condition.} J. Eur. Math. Soc. {\bf 17} (2015), no.~4, 819--835.
\bibitem{era} S.~Ibrahim, N.~Masmoudi and K.~Nakanishi, {\it Correction to the article ``Scattering threshold for the focusing nonlinear Klein-Gordon equation".} Anal. PDE {\bf 9} (2016), no.~2, 503--514.
\bibitem{I} 
M.~Ishiwata, {\it Existence and nonexistence of maximizers for variational problems associated with 
Trudinger-Moser type inequalities in $R^N$.} Math. Ann. {\bf 351} (2011), no.~4, 781--804.
\bibitem{MM}
G.~Mancini and L.~Martinazzi, {\it The Moser-Trudinger inequality and its extremals on a
disk via energy estimates}, Calc. Var. Partial Differential Equations {\bf 56} (2017), no.~4, Art.~94, 26 pp.
\bibitem{MT}
G.~Mancini and P.~-D.~Thizy, {\it Non-existence of extremals for the Adimurthi-Druet inequality}, J. Differential Equations {\bf 266} (2019), no.~2--3, 1051--1072. 
\bibitem{MS}
N.~Masmoudi and F.~Sani, {\it Adams' inequality with the exact growth condition in $\R^4$.} Comm. Pure Appl. Math. {\bf 67} (2014), no.~8, 1307--1335. 
\bibitem{Mo}
J.~Moser, {\em A sharp form of an inequality of N. Trudinger}, Ind. Univ. Math. J. {\bf 20} (1971), 1077--1092. 
\bibitem{Pru}
A.~R.~Pruss, {\it Nonexistence of maxima for perturbations of some inequalities with critical growth}, Canad. Math. Bull. {\bf 39} (1996) no.~2, 227--237. 
\bibitem{Ruf}
B.~Ruf, {\em A sharp Trudinger-Moser type inequality for unbounded domains in $\mathbb R\sp 2$,} J.~Funct.~Anal. {\bf 219} (2005), no.~2, 340--367.
\bibitem{RS}
B.~Ruf and F.~Sani, {\it Ground states for elliptic equations in $\R^2$ with exponential critical growth.}  Geometric properties for parabolic and elliptic PDE's, 251--267, Springer INdAM Ser., 2, Springer, Milan, 2013. 
\bibitem{T}
P.~-D.~Thizy, {\it When does a perturbed Moser-Trudinger inequality admit an extremal?} arxiv:1802.01932. 
\end{thebibliography}
\end{document}